\documentclass[12pt,a4paper]{article}
\usepackage{amssymb}
\usepackage{color}
\usepackage{amsmath}
\usepackage{amsthm}
\usepackage[arrow,matrix,curve]{xy}
\let\mod=\undefined

\DeclareMathOperator{\Ext}{Ext}
\DeclareMathOperator{\Hom}{Hom}
\DeclareMathOperator{\Aut}{Aut}
\DeclareMathOperator{\HomP}{\underline{Hom}}

\DeclareMathOperator{\GL}{GL}
\DeclareMathOperator{\End}{End}
\DeclareMathOperator{\Mat}{Mat}

\DeclareMathOperator{\mod}{mod}
\DeclareMathOperator{\Ld}{Ld}

\DeclareMathOperator{\rank}{rank}
\DeclareMathOperator{\im}{im}

\DeclareMathOperator{\Span}{span}
\DeclareMathOperator{\sub}{sub}
\DeclareMathOperator{\fac}{fac}
\DeclareMathOperator{\parti}{part}
\DeclareMathOperator{\spani}{span}

\newcommand{\op}{{op}}
\newcommand{\N}{\mathbb N}
\newcommand{\Z}{\mathbb Z}

\newcommand{\defeq}{\mathrel{:=}}

\newcommand{\ov}{\overline}

\newtheorem*{mainthm}{Theorem}
\newtheorem{thm}{Theorem}[section]
\newtheorem{cor}[thm]{Corollary}
\newtheorem{lem}[thm]{Lemma}
\newtheorem{prop}[thm]{Proposition}
\newtheorem{remark}[thm]{Remark}

\numberwithin{equation}{section}

\begin{document}

\title{On regular irreducible components of module varieties over string algebras}
\author {M.~Rutscho}
\maketitle

\section*{Abstract}
We determine the regular irreducible components of the variety $\mod(\mathcal A,d)$, where $\mathcal A=kQ/I$ is a string algebra and $I$ is generated by a set of paths of length two. Our case is among the first examples of descriptions of irreducible components, aside from hereditary, tubular (see \cite{GeissSchroer})  and Gelfand-Ponomarev algebras (see \cite{SchroerGelfandPonomarev}).

\section{Introduction}

Fix an algebraically closed field k, and let $\mathcal A$ be a finite dimensional associative $k$-algebra with a unit. By $\mod(\mathcal A,d)$ we denote the affine subvariety of $\Hom_{k}(\mathcal A,M_d(k))$ of $k$-algebra homomorphisms $A \longrightarrow M_d(k)$, where $M_d(k)$ is the algebra of $d\times d$ matrices over $k$. The algebraic group $\GL_d(k)$
of invertible $d\times d$ matrices acts on $\mod(\mathcal A,d)$ by conjugation. The $\GL_d(k)$-orbits in $\mod(\mathcal A,d)$ are in one-to-one correspondence with the isomorphism classes of $d$-dimensional left $\mathcal A$-modules.

We recall the definition of a string algebra and describe its finite dimensional modules. A quiver is a tuple $Q=(Q_0,Q_1,s,t)$ consisting of a finite set of vertices $Q_0$, a finite set of arrows $Q_1$ and maps $s,t:Q_1 \longrightarrow Q_0$. We call $s(\alpha)$ the source and $t(\alpha)$ the target of the arrow $\alpha \in Q_1$.
A path in $Q$ is either $1_u$, where $u$ is a vertex of $Q$, or a finite sequence $\alpha_1\dots \alpha_n$ of arrows of $Q$, satisfying $s(\alpha_{i})=t(\alpha_{i+1})$. The path algebra $kQ$ is the $k$-algebra with $k$-basis $\{p : p \text{ is a path in }Q\}$, where the product of two paths $p,q$ is the concatenation $pq$ in case $s(p)=t(q)$ and zero otherwise.
A string algebra $\mathcal A$ over $k$ is an algebra of the form $\mathcal A=kQ/I$ for a quiver $Q$ and an ideal $I$ in $kQ$ satisfying the following:
\begin{itemize}
\item $I$ is admissible, i.e. there is an $n \in \N$ with $(kQ^+)^n \subseteq I \subseteq (kQ^+)^2$, where
$kQ^+$ is the ideal generated by all arrows.
\item $I$ is generated by a set of paths in $Q$.
\item There are at most two arrows with source $u$ for any vertex $u$ of $Q$.
\item There are at most two arrows with target $u$ for any vertex $u$ of $Q$.
\item For any arrow $\alpha$ of $Q$ there is a most one arrow $\beta$ with $s(\alpha)=t(\beta)$ and $\alpha\beta \notin I$.
\item For any arrow $\beta$ of $Q$ there is a most one arrow $\alpha$ with $s(\alpha)=t(\beta)$ and $\alpha\beta \notin I$.

\end{itemize}

Fix a string algebra $\mathcal A =kQ/I$. The opposite quiver $Q^{\op}$ is $(Q_0,Q_1^{-1},s,t)$, where $Q_1^{-1}= \{\alpha^{-1}: \alpha \in Q_1\}$, $s(\alpha^{-1})\defeq t(\alpha)$ and $t(\alpha^{-1})\defeq s(\alpha)$. A string (of $\mathcal A$) is either $1_u$, where $u$ is a vertex of $Q$, or a finite sequence $\alpha_1\cdots \alpha_n$ of arrows of $Q$ and $Q^{op}$, satisfying $s(\alpha_{i})=t(\alpha_{i+1})$ and $\alpha_i \neq \alpha_{i+1}^{-1}$ such that none of its partial strings $\alpha_i \cdots \alpha_j$  nor its inverse $\alpha_j^{-1}\cdots \alpha_i^{-1}$ belongs to $I$. By $\mathcal W$ we denote the set of all strings.
 Let $c$ be a string. We set $s(c)\defeq t(c) \defeq u$ in case $c=1_u$ and $s(c)\defeq s(\alpha_n)$ and $t(c)=t(\alpha_1)$, in case $c=\alpha_1\cdots \alpha_n$. We say that $c$ starts/ends with an (inverse) arrow if $c=\alpha_1\cdots\alpha_n$ and $\alpha_n \in Q_1$ ($\in Q_1^{-1}$), $\alpha_1 \in Q_1$ ($\in Q_1^{-1}$), respectively.  We define the inverse string  $c^{-1}$ of $c$ by $c^{-1} \defeq 1_u$ in case $c=1_u$ and $c^{-1} \defeq \alpha_n^{-1} \cdots \alpha_1^{-1}$ if $c= \alpha_1 \cdots \alpha_n$. The length $l(c)$ of $c$ is $0$ if $c=1_u$ and $n$ if $c=\alpha_1\cdots \alpha_n$. We call the strings of length $0$ trivial. Note that the concatenation $cd$ of two strings $c$ and $d$ is not necessarily a string.

\begin{remark}
\begin{itemize}
\item[]
\item A string $c$ is trivial if and only if $c=c^{-1}$.
\item If $I$ is generated by paths of length two, then $c=\alpha_1\cdots\alpha_n$ is a string if and only if $\alpha_i\alpha_{i+1}$ is a string for all $1 \leq i \leq n-1$.
\end{itemize}
\end{remark}

By $\Ld(c) \defeq \{c': c=c'c''\}$ we denote the set of leftdivisors of $c$. 
For any string $c$ we define the string module $M(c)$ with basis
\\
$\{e_{c'}: c' \in \Ld(c)\}$ by
\[
p \cdot e_{c'}= \begin{cases} 
e_{c'p^{-1}}&\text{if } c'p^{-1} \in \Ld(c), \\
e_{c''} & \text{if } c'=c''p ,\\
0 & \text{otherwise,}
 \end{cases}
\]
for any path $p$ in $Q$. For an arrow $\alpha$ and a vertex $u$ we thus have
\[
\begin{xy}
\xymatrix{
e_{\alpha_1\cdots \alpha_{i-1}} \ar[r]^{\alpha}& e_{\alpha_1\cdots \alpha_{i} } & \text{if } \alpha=\alpha_i^{-1},\\
e_{\alpha_1\cdots \alpha_{i-1}} & e_{\alpha_1\cdots \alpha_{i} }\ar[l]_/-0.3em/{\alpha} & \text{if } \alpha=\alpha_i,\\
}
\end{xy}
\]
\[
1_u \cdot e_{\alpha_1\cdots\alpha_i}=\begin{cases}e_{\alpha_1\cdots\alpha_i} & \text{if } u=s(\alpha_i),\\
0& \text{if } u\neq s(\alpha_i).
\end{cases}
\]

By \cite{ButlerRingel} string modules are indecomposable and two string modules $M(c)$ and $M(d)$ are isomorphic if and only if $c=d$ or $c=d^{-1}$.
An isomorphism $M(c)\longrightarrow M(c^{-1})$ is given by sending $e_{c'}$ to $e_{c''^{-1}}$ for $c' \in \Ld(c)$ with $c=c'c''$. We will refer to such an isomorphism as "the isomorphism from $M(c)$ to $M(c^{-1})$".

Aside from string modules there is another type of indecomposable (finite dimensional) $\mathcal A$-modules, the band modules. To make it easier to describe degenerations (see section \ref{deg}), we also define quasi-band modules, which are a generalization of band modules.

A quasi-band $(b,m)$ is a map $b:\Z \longrightarrow Q_1 \cup Q_1^{-1}$ together with an integer $m \geq 1$ such that $b(i)=b(i+m)$ for all $i \in  \Z$ and $b(i)b(i+1)\cdots b(i+n)$ is a string for all $i \in \Z$ and all $n \geq 0$. Frequently we will just write $(b,m)=b(1)\cdots b(m)$.  A quasi-band $(b,m)$ is called a band provided $(b,m')$ is not a quasi-band for any $0<m'<m$.
For any quasi-band $(b,m)$ and any $\phi \in \Aut_k(V)$, where $V$ is a finite dimensional $k$-vector space, we define the quasi-band module $M(b,m,\phi)$ in the following way. First we define an (infinite dimensional) $\mathcal A$-module $M(b)$ with basis $\{e_i: i \in \Z\}$ by
\[
p  \cdot e_{i}= \begin{cases} 
 e_{j}& \text{if there is a } j \geq i \text{ such that } b(i)p^{-1}=b(i)\cdots b(j),\\
 e_{j-1} & \text{if there is a } j \leq i+1 \text{ such that } pb(i+1)=b(j)\cdots b(i+1),\\
0 & \text{otherwise,}
 \end{cases}
\]
for any path $p$ in $Q$. Note that we write $pb(i+1)=b(j)\cdots b(i+1)$ instead of $p=b(j) \cdots b(i)$ in order to include trivial paths. For an arrow $\alpha$ and and a vertex $u$ we thus have
\[
\begin{xy}
\xymatrix{
e_{i-1} \ar[r]^\alpha &e_{i} & \text{if } \alpha=b(i)^{-1},\\
e_{i-1}  &e_{i}\ar[l]^\alpha & \text{if } \alpha=b(i),\\
}
\end{xy}
\]
\[
1_u \cdot e_{i}=\begin{cases}e_{i} & \text{if } u=s(b(i)),\\
0& \text{if } u\neq s(b(i)).
\end{cases}
\]

We define an $\mathcal A$-module structure on $V \otimes_k M(b)$ by setting
\[p\cdot (v \otimes w)\defeq v \otimes (p \cdot w)\]
 for any path $p$ in $Q$. Finally we set 
\[
M(b,m,\phi) \defeq V  \otimes_k M(b)/\Span_k(\{v \otimes e_i -\phi(v) \otimes e_{i+m}: v \in V, i  \in \Z\}).
\]
In case $V=k$ the automorphism $\phi$ is given by multiplication with a $\lambda \in k^*$ and we set $M(b,m,\lambda)=M(b,m,\phi)$.
We call $M(b,m,\phi)$ a band module provided $(b,m)$ is a band and the $k[x]$-module defined by $\phi$ is indecomposable. By \cite{ButlerRingel} any band module is indecomposable and two band modules $M(b,m,\phi)$ and $M(b',m',\phi')$ are isomorphic if and only if $m=m'$ and one of the following holds:
\begin{itemize}
 \item There is an $i \in \Z$ with $b(j)=b'(i+j)$ for all $j \in \Z$ and $\phi$ and $\phi'$ are isomorphic as $k[x]$-modules.
 \item There is an $i \in \Z$ with $b(j)=b'(i-j)^{-1}$ for all $j \in \Z$ and $\phi^{-1}$ and $\phi'$ are isomorphic as $k[x]$-modules.
\end{itemize}
This motivates the definition of an equivalence relation $\sim$ for quasi-bands, defined by $(b,m) \sim (b',m')$ if $m=m'$ and one of the following holds:
\begin{itemize}
 \item There is an $i \in \Z$ with $b(j)=b'(i+j)$ for all $j \in Z$.
 \item There is an $i \in \Z$ with $b(j)=b'(i-j)^{-1}$ for all $j \in Z$.
\end{itemize}
By $[(b,m)]$ we denote the equivalence class of $(b,m)$ with respect to $\sim$.

It is shown in \cite{ButlerRingel} that the finite-dimensional indecomposable $\mathcal A$-modules  are precisely the string and band modules up to isomorphism.

For any sequence $S=(c_1,\ldots,c_l,(b_1,m_1),\cdots (b_n,m_n))$ with $l,n \geq 0$ consisting of strings $c_1,\ldots,c_l$ and quasi-bands $(b_1,m_1),\cdots,(b_m,m_n)$
the family of modules $\mathcal F(S) \subseteq  \mod(\mathcal A,d)$ is the image of the morphism
\[\GL_d(k) \times (k^*)^n \longrightarrow \mod(\mathcal A,d)\] sending $(g,\lambda_1,\ldots,\lambda_n)$ to
\[g \star (\bigoplus\limits_{i=1}^l M(c_i) \oplus \bigoplus \limits_{j=1}^n M(b_j,m_j,\lambda_j)),\] where \[d= \sum \limits_{i=1}^l \dim_k M(c_i) + \sum \limits_{j=1}^n \dim_k M(b_j,m_j,1).\]

We call a subset $\mathcal F$ of $\mod(\mathcal A,d)$ an $\mathcal S$-family of strings and quasi-bands if there is a sequence $S$ of strings and quasi-bands with $\mathcal F=\mathcal F(S)$ and we call $\mathcal F$ an $\mathcal S$-family of (strings and) bands if $S$ is a sequence of (strings and) bands.

Note that a band module $M(b,\phi)$ with $\phi \in \GL_p(k)=\Aut_k(k^p)$  does not necessarily belong to any $\mathcal S$-family of strings and quasi-bands, as $\phi$ might not be diagonalizable. But, as the set of diagonalizable matrices in $\GL_p(k)$ is dense in $\GL_p(k)$, we see that $M(b,\phi)$ belongs to the closure of the $\mathcal S$-family $\mathcal F(b,b,\ldots,b)$, which is an $\mathcal S$-family of bands. Thus $\mod( \mathcal A,d)$ is a union of closures of $\mathcal S$-families of strings and bands. As $\GL_d(k)$ is irreducible, any $\mathcal S$-family of strings and quasi-bands is irreducible and as there are only finitely many different $\mathcal S$-families of strings and bands in $\mod(\mathcal A,d)$, any irreducible component of $\mod(\mathcal A,d)$ is the closure of an $\mathcal S$-family of strings and bands.

Let $r:\mod(\mathcal A,d)\longrightarrow \N$ be the function sending $X$ to
\[
r(X) =\sum\limits_{\alpha \in Q_1}\rank X(\alpha).
\]
We call $X\in \mod(\mathcal A,d)$ regular if $r(X)=d$. 
From the direct decomposition of $X$ into a direct sum of string and band modules we obtain that $r(X) \leq d$ for any $X \in \mod(\mathcal A,d)$ and that $X$ is regular if and only if $X$ is isomorphic to a direct sum of band modules. As $r$ is lower semi-continuous, we see that the regular elements of $\mod(\mathcal A,d)$ form an open subset of $\mod(\mathcal A,d)$.

Let $\mathcal C$ be an irreducible component of $\mod(\mathcal A,d)$ such that there is a regular $X \in \mathcal C$. We call such an irreducible component regular. We already know that there is a sequence of strings and bands $S$ such that the closure of $\mathcal F(S)$ is $\mathcal C$. Obviously $S$ has to be a sequence of bands. In order to determine the regular irreducible components of $\mod(\mathcal A,d)$ it suffices to solve the following problem: Given a sequence of bands $S$ with $\mathcal F(S) \subseteq \mod(\mathcal A,d)$, determine whether the closure of $\mathcal F(S)$ is an irreducible component or not.

We apply the result on decompositions of irreducible components as presented in \cite{CBJS} and obtain the following: Let $S=(b_1,\ldots,b_n)$ be a sequence of bands. We set $d_i \defeq \dim_k M(b_i,1)$ and $d=d_1+\ldots+d_n$. The closure of $\mathcal F(S)$ is an irreducible component of $\mod(\mathcal A,d)$ if and only if the following holds:
\begin{itemize}
\item[i)] For all $i \neq j$, there are $X \in \mathcal F(b_i)$ and $Y \in \mathcal F(b_j)$ with  $\Ext^1_{\mathcal A}(X,Y)=0$.
\item[ii)] The closure of $\mathcal F(b_i)$ is an irreducible component of $\mod(\mathcal A,d_i)$ for $i=1,\ldots,n$.
\end{itemize}

Our goal is to characterize the conditions $i)$ and $ii)$ by combinatorial criteria on bands. For $i)$ we have a complete solution, whereas our characterization of $ii)$ only holds if the ideal $I$ is generated by paths of length two.
Note that the result from \cite{CBJS} can also be applied to sequences of strings and bands in order to determine the non-regular irreducible components, but we were not able to characterize condition $ii)$ for strings.

We call a pair of bands $((b,m),(c,n))$ extendable if there are $s,t \geq 1$, strings $u,v,w$ and arrows $\alpha,\beta,\gamma,\delta$ with
\[
(b,sm)=w\beta u \alpha^{-1} \text{ and } (c,tn)=w\delta^{-1} v\gamma
\]
 such that
\[(d,n+m)\defeq(c,n)(b,m) \defeq c(1)\cdots c(n) b(1) \cdots b(m)\] is a quasi-band. Note that $l(w)>m,n$ is possible, which explains why $s$ and $t$ are needed. We call a pair of equivalence classes of bands $(B,C)$ extendable, if there are bands $(b,m) \in B$ and $(c,n) \in C$ such that $((b,m),(c,n))$ is extendable.

\begin{prop} \label{prop1}Let $(b,m)$ and $(c,n)$ be bands. There are $X \in \mathcal F(b,m)$ and $Y \in \mathcal F(c,n)$ with $\Ext^1_{\mathcal A}(X,Y)=0$ if and only if the pair $([(b,m)],[(c,n)])$ is not extendable.
\end{prop}
Note that $\Ext^1_{\mathcal A}(X,Y)=0$ for some $X \in \mathcal F(b,m)$ and $Y \in \mathcal F(c,n)$ implies that $\Ext^1_{\mathcal A}(-,-)$ vanishes generically on 
$\mathcal F(b,m)\times \mathcal F(c,n)$.

We call a band $(b,m)$ negligible if one of the following holds:
\begin{itemize}
\item There are strings $u,v,w,x,y$, arrows $\alpha,\beta,\gamma,\delta$ and an $s\geq 1$ with 
\[(b,m)=u\gamma v \alpha^{-1} \text{ and }
(b,sm)=w \beta x \alpha^{-1}=u \gamma w \delta^{-1} y\]
such that
\[(c,n) \defeq u\gamma \text{ and }  (d,m-n) \defeq v\alpha^{-1}\] are quasi-bands.

\item  There is a string $u$ that starts and ends with an arrow,  a string $v$ that starts and ends with an inverse arrow and a string $w$ with $(b,m)=wuw^{-1}v$ such that
\[
(c,m) \defeq w u^{-1} w^{-1} v
\]
is a quasi-band.

\end{itemize}

We call an equivalence class of bands $B$ negligible if there is a band $(b,m) \in B$ which is negligible. One can show that $(B,B)$ is extendable if $B$ is negligible, but we will not use it.

\begin{prop} \label{prop2} Let $(b,m)$ be a band with $\mathcal F(b,m) \subseteq \mod(\mathcal A,d)$. If the closure of $\mathcal F(b,m)$ is an irreducible component of $\mod(\mathcal A,d)$, then $[(b,m)]$ is not negligible.
\end{prop}

We do not know whether the converse holds in general, but it does in case $I$ is generated paths of length two:

\begin{prop} \label{prop3} Assume that $I$ is generated by a set of paths of length two and let $(b,m)$ be a band with $\mathcal F(b,m) \subseteq \mod(\mathcal A,d)$. If $[(b,m)]$ is not negligible, then the closure of $\mathcal F(b,m)$ is an irreducible component of $\mod(\mathcal A,d)$.

\end{prop}

We call a sequence $S=(b_1,\ldots,b_n)$ of bands negligible, if one of the following holds:
\begin{itemize}
\item $[b_i]$ is negligible for some $1 \leq i \leq n$.
\item $([b_i],[b_j])$ is extendable for some $1 \leq i,j \leq n$.
\end{itemize}

Our main result is the following theorem which is just a consequence from the previous propositions.

\begin{mainthm} Let $\mathcal A=kQ/I$ be a string algebra and let $S$ be a sequence of bands with $\mathcal F(S) \subseteq \mod(\mathcal A,d)$.
\begin{itemize}
\item[a)] If the closure of $\mathcal F(S)$ is an irreducible component of $\mod(\mathcal A,d)$, then $S$ is negligible.
\item[b)] If $S$ is negligible and $I$ is generated by paths of length two, then the closure of $\mathcal F(S)$ is an irreducible component of $\mod(\mathcal A,d)$.

\end{itemize}

\end{mainthm}

If $I$ is generated by a set of paths of length two and $b$ is a band, then $[b]$ is negligible if and only if $([b],[b])$ is extendable (see Lemma \ref{charnegligible} and \ref{charextendable}). 
From the previous theorem we thus obtain:

\begin{cor}\label{cor} Assume that $I$ is generated by paths of length two and let $\mathcal F \subseteq \mod(\mathcal A,d)$ be an $\mathcal S$-family of bands. The closure of $\mathcal F$ is an irreducible component of $\mod(\mathcal A,d)$ if and only if there are $X,Y \in \mathcal F$ with $\Ext^1_{\mathcal A}(X,Y)=0$.

\end{cor}

Note that Corollary \ref{cor} is not true if $I$ is not generated by paths of length two. Indeed, consider the algebra $\Lambda =k[\alpha,\beta]/( \alpha^3,\beta^3,\alpha\beta)$ and the band $\alpha^{-1}\beta$.
The closure of $\mathcal F(\alpha^{-1}\beta)$ is an irreducible component of $\mod(\Lambda,2)$, as there are no other $\mathcal S$-families of band modules in $\mod(\Lambda,2)$. On the other hand, $\Ext^1_{\Lambda}(X,Y)$ does not vanish for any $X,Y \in \mathcal F(\alpha ^{-1}\beta)$, as one can easily construct a short exact sequence $0 \longrightarrow X \longrightarrow Z \longrightarrow Y \longrightarrow 0$ for some $Z \in \mathcal F(\alpha^{-2}\beta^2)$.

If $I$ is generated by paths of length two, there is a simple formula for the dimension of a regular irreducible component:
We call a vertex $u \in Q_0$ gentle (w.r.t. $I$)  if it satisfies the following:
\begin{itemize}
\item For any arrow $\alpha$ of $Q$ with $s(\alpha)=u$ there is a most one arrow $\beta$ with $s(\alpha)=t(\beta)$ and $\alpha\beta \in I$.
\item For any arrow $\beta$ of $Q$ with $t(\beta)=u$ there is a most one arrow $\alpha$ with $s(\alpha)=t(\beta)$ and $\alpha\beta \in I$.
\end{itemize}
We call $\mathcal A$ a gentle algebra if  any vertex of $Q$ is gentle and the ideal $I$ is generated by paths of length two.

\begin{prop}\label{propdim} Assume that $I$ is generated by paths of length two and let $S$ be a sequence of bands such that the closure $\mathcal F(S)$ is an  irreducible component of $\mod(\mathcal A,d)$. The dimension of $\ov{\mathcal F(S)}$ is given by the formula
\[
\dim \ov{\mathcal F(S)} = d^2 -  \sum \limits_{\substack{u \in Q_0\\ u \text{ non-gentle}}}\dim_k \Hom_{\mathcal A} (X,M(1_u)) \dim_k\Hom_{\mathcal A}(M(1_u),X),
\]
for any $X \in \mathcal F(S)$. In particular, the dimension of any regular irreducible component of $\mod(\mathcal A,d)$ is $d^2$ provided $\mathcal A$ is a gentle algebra.

\end{prop}

The paper is organized as follows. In section \ref{hom} we recall results on homomorphisms between representations of string algebras from \cite{Krause}. In section \ref{proofof3anddim} we prove  Proposition \ref{prop3} and the dimension formula Proposition \ref{propdim}. Section \ref{deg} is devoted to explicit inclusions among closures of $\mathcal S$-families of bands and the proof of Proposition \ref{prop2}. In section \ref{ext} we study extensions and prove Proposition \ref{prop1}.

\newpage
\section{Homomorphisms}\label{hom}

We show that $\mathcal S$-families of band modules can be separated by hom-conditions using string modules (see Proposition \ref{separate}), a result we need for the proof of Proposition \ref{prop3}. We recall a basis of homomorphism spaces between representations of string algebras worked out in \cite{Krause}.

\subsection{Substring morphisms for a string}
Let $c$ be a string. A substring of $c$ is a triple $(c_1,c_2,c_3)$ of strings with $c=c_1c_2c_3$ satisfying the following:
\begin{itemize}
\item $c_1$ is either trivial or it starts with an inverse arrow ($c_1=c_1'\alpha^{-1}$).
\item $c_3$ is either trivial or it ends with an arrow ($c_3=\alpha c_3'$).
\end{itemize}
By $\sub(c)$ we denote the set of substrings of $c$.

For each $(c_1,c_2,c_3) \in \sub(c)$ we define the homomorphism 
\[
\iota_{c_2,(c_1,c_2,c_3)}:M(c_2) \longrightarrow M(c)
\]
by sending $e_{d}$ to $e_{c_1d}$ for $d \in \Ld(c_2)$. We call such a morphism a substring morphism. For a string $d$ we set
\[
\sub(d,c) \defeq \{(c_1,c_2,c_3) \in \sub(c): c_2 \in \{d,d^{-1}\}\}.
\]

\subsection{Factorstring morphisms for a string}
Let $c$ be a string. A factorstring of $c$ is a triple $(c_1,c_2,c_3)$ of strings with $c=c_1c_2c_3$ satisfying the following:
\begin{itemize}
\item $c_1$ is either trivial or it starts with an arrow ($c_1=c_1'\alpha$).
\item $c_3$ is either trivial or it ends with an inverse arrow ($c_3=\alpha^{-1}c_3'$).
\end{itemize}
By $\fac(c)$ we denote the set of substrings of $c$.

For each $(c_1,c_2,c_3) \in \fac(c)$ we define the homomorphism 
\[
\pi_{c_2,(c_1,c_2,c_3)}:M(c) \longrightarrow M(c_2)
\]
by sending $e_{d}$ to 
\[
\begin{cases}
e_{d'} &\text{if } d=c_1d' \in \Ld(c_1c_2)\\
0 &\text{otherwise.}\\
\end{cases}
\]
We call such a morphism a factorstring morphism. For a string $d$ we set
\[
\fac(d,c) \defeq \{(c_1,c_2,c_3) \in \fac(c): c_2 \in \{d,d^{-1}\}\}.
\]

\subsection{Winding and unwinding morphisms}
Let $(b,m)$ be a quasi-band. For $s \geq 1$ and $\lambda \in k^*$ we define the winding morphism
\[
w_{(b,m),s,\lambda}: M(b,sm,\lambda^s) \longrightarrow M(b,m,\lambda)
\]
by sending $\ov{1 \otimes e_i}$ to $\ov{1\otimes e_i}$ for $i=0,\ldots,sm-1$.

Dually, the unwinding morphism
\[
u_{(b,m),s,\lambda}: M(b,m,\lambda) \longrightarrow M(b,sm,\lambda^s),
\]
sends $\ov{1\otimes e_i}$ to
\[
\sum\limits_{j=0}^{s-1} \lambda^j \ov{1\otimes e_{i+jm}}
\]
for $i=0,\ldots,m-1$.

\subsection{Substring morphisms for a quasi-band}\label{substringsofbands}
Let $(b,m)$ be a quasi-band and $c$ a string. We define the set
\begin{eqnarray*}
\sub_1(c,(b,m)) &\defeq& \{1 \leq i \leq m:  b(i) \cdots b(i+l(c))=b(i)c,\\
&&b(i)^{-1}, b(i+l(c)+1) \in Q_1\}
\end{eqnarray*}
Note that we write $b(i)\cdots b(i+l(c))=b(i)c$ instead of $b(i+1)\cdots b(i+l(c))=c$ in order to include the case $l(c)=0$.
For any $i \in \sub_1(c,(b,m))$ and any $\lambda \in k^*$ we define a morphism 
\[
\iota_{i,c,(b,m),\lambda}: M(c) \longrightarrow M(b,m,\lambda)
\]
as a composition
\[
\begin{xy}
\xymatrix{
M(c) \ar[r]^/-0.8em/f &M(b,sm,\lambda^s) \ar[rr]^{w_{(b,m),s,\lambda}}&& M(b,m,\lambda)
}
\end{xy}
\]
for some integer $s\geq 1$ chosen in such a way that  $(b,sm)=d_1 \alpha^{-1} c \beta d_2$ for some arrows $\alpha,\beta$ and some strings $d_1,d_2$ with $l(d_1)=i-1$, where the morphism $f$ sends $e_d$ to $\ov{1\otimes e_{i+l(d)}}$ for $d \in \Ld(c)$.

\[
\begin{xy}
\xymatrix{
M(c): & e_{1_{t(c)}} \ar@{-}[rrr]^{c}\ar@{.>}[d]&\ar@{.>}[d]&\ar@{.>}[d] & e_{c}\ar@{.>}[d]\\
&\ov{1\otimes e_{i}}\ar@{-}[rrr]^{c}& && \ov{1\otimes e_{i+l(c)}}\\
M(b,sm,\lambda^s):\\
&\ov{1\otimes e_{i-1-sm}} \ar@{-}[rrr]^{d_2d_1}\ar[uu]^{\alpha=\lambda^{-s}} &&& \ov{1\otimes e_{i+l(c)+1}}\ar[uu]_{\beta}\\
}
\end{xy}
\]
Note that $ \iota_{i,c,(b,m),\lambda}$ does not depend on the choice of $s$.

We set $\sub_{-1}(c,(b,m)) \defeq \sub(c^{-1},(b,m))$. Note that $\sub_{-1}(c,(b,m))=\sub_{1}(c,(b,m))$ if $c$ is trivial. For each $i \in \sub_{-1}(c,(b,m))$ we define a morphism
\[
\iota_{i,c,(b,m),\lambda}: M(c) \longrightarrow M(b,m,\lambda).
\]
as the composition
\[
\begin{xy}
\xymatrix{
M(c) \ar[r]^/-0.4 em/{\sim} &M(c^{-1}) \ar[rr]^/-0.4 em/{\iota_{i,c^{-1},(b,m),\lambda}}&& M(b,m,\lambda),
}
\end{xy}
\]
where the first morphism is the isomorphism from $M(c)$ to $M(c^{-1})$.

We have thus defined morphisms $\iota_{i,c,(b,m),\lambda}$, called substring morphisms, for any $\lambda \in k^*$ and any 
\[
i \in \sub(c,(b,m)) \defeq \sub_{1}(c,(b,m))\cup \sub_{-1}(c,(b,m)).
\]
Whereas substring morphisms for strings are always injective, substring morphisms for quasi-bands are not necessarily.

Dually we define the factorstring morphisms for quasi-bands: 

\subsection{Factorstring morphisms for a quasi-band}

Let $(b,m)$ be a quasi-band and $c$ a string. We define the set
\begin{eqnarray*}
\fac_1(c,(b,m)) &\defeq& \{1 \leq i \leq m: b(i) \cdots b(i+l(c))=b(i)c,\\
&& b(i), b(i+l(c)+1)^{-1} \in Q_1 \}
\end{eqnarray*}
For any $i \in \fac_1(c,(b,m))$ and any $\lambda \in k^*$ we define a morphism 
\[
\pi_{i,c,(b,m),\lambda}: M(b,m,\lambda) \longrightarrow M(c)
\]
as a composition
\[
\begin{xy}
\xymatrix{
M(b,m,\lambda) \ar[rr]^{u_{(b,m),s,\lambda}}&&M(b,sm,\lambda^s) \ar[r]^/0.8em/{f}& M(c)
}
\end{xy}
\]
for some integer $s\geq 1$ chosen in such a way that  $(b,sm)=d_1 \alpha c \beta^{-1} d_2$ for some arrows $\alpha,\beta$ and some strings $d_1,d_2$ with $l(d_1)=i-1$, where $f$ is the morphism that sends $\ov{1\otimes e_{j}}$ to $0$ if either $0 \leq j<i$ or $i+l(c) <j<sm$, and $\ov{1\otimes e_j}$ to
$e_{d}$ if $i \leq j \leq i +l(c)$, where $d$ is the leftdivisor of $c$ of length $j-i$.

We set $\fac_{-1}(c,(b,m)) \defeq \fac(c^{-1},(b,m))$.
For each $i \in \fac_{-1}(c,(b,m))$ we obtain a morphism 
\[
\pi_{i,c,(b,m),\lambda}: M(b,m,\lambda) \longrightarrow M(c)
\]
by identifying $M(c)$ and $M(c^{-1})$ just as above.

We have thus defined morphisms $\pi_{i,c,(b,m),\lambda}$, called factorstring morphisms, for any $\lambda \in k^*$ and any 
\[
i \in \fac(c,(b,m)) \defeq \fac_{1}(c,(b,m))\cup \fac_{-1}(c,(b,m)).
\]

\subsection{Morphisms between string and band modules}

In this section we will frequently use the abbreviation
\[
[X,Y]\defeq \dim_k \Hom_{\mathcal A}(X,Y)
\]
for $\mathcal A$-modules $X,Y$.
The following three propositions are reformulations of results from \cite{Krause}.

\begin{prop}\label{band->string} Let $M(b,m,\lambda)$ be a band module and $M(c)$ a string module. The morphisms
\[
\iota_{d,x} \circ \pi_{i,d,(b,m),\lambda}:M(b,m,\lambda) \longrightarrow M(c), 
\]
where $d$ is a string of length at most $l(c)$, $x \in \sub(d,c)$ and $i \in \fac(d,(b,m))$, form a basis of $\Hom_{\mathcal A}(M(b,m,\lambda),M(c))$. In particular,
\[
[M(b,m,\lambda),M(c)] = \sum \limits_{d \in \mathcal W, l(d) \leq l(c)} \sharp \fac(d,(b,m)) \sharp \sub(d,c).
\]

\end{prop}

\begin{prop}\label{string->band} Let $M(b,m,\lambda)$ be a band module and $M(c)$ a string module. The morphisms
\[
\iota_{i,d,(b,m),\lambda}\circ \pi_{d,x}:M(c)\longrightarrow M(b,m,\lambda), 
\]
where $d$ is a string of length at most $l(c)$, $x \in \fac(d,c)$ and $i \in \sub(d,(b,m))$, form a basis of $\Hom_{\mathcal A}(M(c),M(b,m,\lambda))$. In particular,
\[
[M(c),M(b,m,\lambda)] = \sum \limits_{d \in \mathcal W, l(d) \leq l(c)}\sharp \fac(d,c) \sharp \sub(d,(b,m)) .
\]
\end{prop}

\begin{prop}\label{band->band} Let $M(b,m,\lambda)$ and $M(c,n,\mu)$ be band modules. The morphisms
\[
\iota_{j,d,(c,n),\mu} \circ \pi_{i,d,(b,m),\lambda}: M(b,m,\lambda) \longrightarrow M(c,n,\mu),
\]
where $d$ is a string, $j \in \sub(d,(c,n))$ and $i \in \fac(d,(b,m))$ (together with an isomorphism in case $M(b,m,\lambda)$ and $M(c,n,\mu)$ are isomorphic) form a basis of $\Hom_{\mathcal A}(M(b,m,\lambda),M(c,n,\mu))$.
\end{prop}

\noindent
As an example, we present a result which we will need in section \ref{ext}.
\begin{lem}\label{?->injective} Let $M(b,m,\lambda)$ and $M(c,n,\mu)$ be band modules, $d$ a string,  $j \in \sub(d,(c,n))$ and $i \in \fac(d,(b,m))$.
The morphism
\[
\iota_{j,d,(c,n),\mu} \circ \pi_{i,d,(b,m),\lambda}: M(b,m,\lambda) \longrightarrow M(c,n,\mu)
\]
is injective  if $m \leq l(d)< n+m$ and $m< n$, and it is surjective if $n \leq l(d) < m+n$ and $n<m$.
\end{lem}

Note that Lemma \ref{?->injective} may become wrong if we drop the condition $l(d) < m+n$.

\begin{proof} We may assume that $j=n$ and $i=m$. Up to duality it suffices to show that the morphism
\[
f \defeq \iota_{j,d,(c,n),\mu} \circ \pi_{i,d,(b,m),\lambda}
\]
is injective if $m   \leq l(d) < n+m$ and $m <n$. Let $A$ be the matrix of $f$ with respect to the bases $\ov{1\otimes e_0},\ldots,\ov{1\otimes e_{m-1}}$ of $M(b,m,\lambda)$ and 
$\ov{1\otimes e_0},\ldots,\ov{1\otimes e_{n-1}}$ of $M(c,n,\mu)$. 
 To show that $f$ is injective, we list the possible forms of $A$ depending on the relation between $n,m$ and $l(d)$:
\\
If $l(d) <n$, then $A$ is of the form
\[\left(\begin{array}{c}
A_1\\
A_2
\end{array}\right),\] where $A_1 \in \Mat(m\times m,k)$ is the identity matrix and $A_2\in \Mat(n-m\times m,k)$. From now on we assume that $n\leq l(d)$. If $2m\leq n$, then $A$ is of the form
\[\left(\begin{array}{c}
A_1\\
A_2\\
A_3
\end{array}\right),\] where $A_1,A_2 \in\Mat(m\times m,k)$, $A_3 \in \Mat(n-2m\times m,k)$ and $A_2=\lambda\cdot 1_{m\times m}$ is a  multiple of the identity matrix. Finally, we assume that $n <2m$. We decompose $A=B+C$, where
\[
B=\left(\begin{array}{cc}
1_{n-m\times n-m} & 0_{n-m \times 2m-n}  \\
0_{2m-n\times n-m} & 1_{2m-n \times 2m-n}  \\
\lambda \cdot 1_{n-m\times n-m} & 0_{n-m \times 2m-n}  \\
\end{array}\right),
\]
\[
C=\left(\begin{array}{cc}
0_{2m-n\times n-m} & C_1  \\
C_2 & 0_{n-m \times 2m-n}  \\
0_{n-m \times n-m} & 0_{n-m \times 2m-n}  \\
\end{array}\right)
\]
and $C_1$ and $C_2$ are diagonal matrices. Note that the sizes of the blocks in $B$ and $C$ are not necessarily the same.
Now we see that $A$ is of the form
\[\left(\begin{array}{cc}
A_{11} &A_{12}\\
A_{21}&A_{22}\\
A_{31}& A_{32}
\end{array}\right),\] where
\begin{itemize}
\item $A_{11},A_{31} \in \Mat(n-m\times n-m,k)$,
\item $A_{12},A_{32} \in \Mat(n-m\times 2m-n)$,
\item $A_{21} \in \Mat(2m-n\times n-m,k)$,
\item $A_{22} \in \Mat(2m-n\times 2m-n)$,
\item $A_{22}$ is upper triangular and all its entries on the diagonal are $1$,
\item $A_{31}=\lambda \cdot 1_{n-m\times n-m}$  and
\item $A_{32}$ is zero.
\end{itemize}

\end{proof}

\noindent
The following proposition shows that the functions 
\[
[M(c),-],[-,M(c)]:\mod(\mathcal A,d) \longrightarrow \N
\] for $c \in \mathcal W$ separate $\mathcal S$-families of band modules.

\begin{prop}\label{separate} Let $S=(b_1,\ldots,b_m)$ and $T=(c_1,\ldots,c_n)$ be sequences of bands with $\mathcal F(S),\mathcal F(T) \subseteq \mod(\mathcal A,d)$.
If
\[
[M(c),X]=[M(c),Y] \text{ and } [X,M(c)]=[Y,M(c)]
\]
 for any string $c$, $X \in \mathcal F(S)$ and $Y \in \mathcal F(T)$, then
\[
\mathcal F(S)=\mathcal F(T).
\]
\end{prop}
We do not claim that Proposition \ref{separate} holds for sequences of quasi-bands.

Let $X,Y$ be $\mathcal A$-modules. Clearly the dimension vectors of $X$ and $Y$ coincide (i.e. $\dim_k \im X(1_u)=\dim_k \im Y(1_u)$ for all $u\in Q_0$) if and only if $[P,X]=[P,Y]$ for any projective $\mathcal A$-module $P$ if and only if $[X,J]=[Y,J]$ for any injective $\mathcal A$-module $J$. Recall from \cite{Auslander-Reiten}, that 
\[
[U,X] -[X,\tau U]=[P_0,X]-[P_1,X],
\]
 where $P_1 \longrightarrow P_0 \longrightarrow U \longrightarrow 0$ is a minimal projective presentation of an $\mathcal A$-module $U$ and $\tau$ denotes the Auslander-Reiten translation.  Dually, if $0 \longrightarrow U \longrightarrow J_0 \longrightarrow J_1$ is a minimal injective copresentation of $U$, then
\[
[X,U] -[\tau^{-} U,X]=[X,J_0]-[X,J_1].
\]
As the Auslander-Reiten translate of a string module is either $0$ or a string module (see \cite{ButlerRingel}) and as all projective and injective $\mathcal A$-modules are string modules, we obtain the following corollary.

\begin{cor} Let $S$ and $T$ be sequences of bands with $\mathcal F(S),\mathcal F(T)\subseteq \mod(\mathcal A,d)$ and let $X \in \mathcal F(S)$ and $Y \in \mathcal F(T)$. The following are equivalent:
\begin{itemize}
\item[i)] $\mathcal F(S)=\mathcal F(T)$

\item[ii)] $[M(c),X]=[M(c),Y]$ for any string $c$.

\item[iii)] $[X,M(c)]=[Y,M(c)]$ for any string $c$.

\end{itemize}
\end{cor}

Before we can prove Proposition \ref{separate} we need some additional definitions and a technical lemma.
For any non-trivial string $c$ and any quasi-band $(b,m)$ we set
\begin{itemize}
\item $\parti_1(c,(b,m)) \defeq \{1 \leq i \leq m: b(i) \cdots b(i+l(c)-1) =c\},$

\item $\parti_{-1}(c,(b,m)) \defeq \parti_1(c^{-1},(b,m))$ and 

\item $ \parti(c,(b,m)) \defeq \parti_1(c,(b,m)) \cup \parti_{-1}(c,(b,m)).$
\end{itemize}

We extend the definition of $\sub(c,-),\fac(c,-)$ for a string $c$ and $\parti(c,-)$ for a non-trivial string $c$ to sequences of quasi-bands instead of a single quasi-band. For a sequence $S=(b_1,\ldots,b_n)$ of bands, we set
\begin{itemize}
\item $\parti(c,S) \defeq \bigcup\limits_{i=1}^n (\parti(c,b_i)\times \{i\}) \subseteq \N \times \N$

\item $\sub(c,S) \defeq \bigcup\limits_{i=1}^n (\sub(c,b_i)\times \{i\})\subseteq \N \times \N$

\item $\fac(c,S) \defeq \bigcup\limits_{i=1}^n (\fac(c,b_i)\times \{i\})\subseteq \N \times \N$

\end{itemize}
Moreover, we define
\[
[c,S]\defeq \sum\limits_{d \in \mathcal W,l(d)\leq l(c)} \sharp \fac(d,c) \sharp \sub(d,S)
\]
and
\[
[S,c]\defeq \sum\limits_{d \in \mathcal W,l(d)\leq l(c)} \sharp \fac(d,S)\sharp \sub(d,c).
\]

As direct consequence of Proposition \ref{band->string} and Proposition \ref{string->band} we obtain

\begin{cor}\label{calcdimhomseq} Let $S$ be a sequence of bands and $X \in \mathcal F(S)$. For any string $c$ we have
\[
[c,S]=[M(c),X] \text{ and } [S,c]=[X,M(c)].
\]
\end{cor}

We come to the technical lemma.

\begin{lem} \label{techlem} Let $S$ and $T$ be sequences of bands with $\rank X(\alpha)= \rank Y(\alpha)$ for any arrow $\alpha$, $X \in \mathcal F(S)$ and $Y \in \mathcal F(T)$ and let $N \in \N$.
 If $[c,S]=[c,T]$ and $[S,c]=[T,c]$ for any string $c$ of length at most $N$, then
$\sharp \parti(d,S)= \sharp \parti(d,T)$ for any non-trivial string $d$ of length at most $N+2$.
\end{lem}

\begin{proof} It follows from Corollary \ref{calcdimhomseq} that $\sharp \sub(c,S) = \sharp \sub(c,T)$ and \\
$\sharp \fac(c,S)=\sharp \fac(c,T)$ for any string $c$ of length at most $N$.
We will prove that $\sharp \parti(d,S)= \sharp \parti(d,T)$ for $1 \leq l(d) \leq N+2$ by induction on the length of $d$. If $l(d)=1$, we may assume that $d$ is an arrow. Let $X \in \mathcal F(S)$ and $Y \in \mathcal F(T)$. We have
\[
\sharp \parti(d,S)=\rank X(d)=\rank Y(d) = \sharp\parti(d,T).
\]
If $N+2\geq l(d)>1$, then $d$ is of the form $d=d_1 c d_2$ for a (possibly trivial) string $c$ of length at most $N$ and strings $d_1,d_2$ of length one. We assume that $\sharp \parti (d,S) \neq  \sharp\parti(d,T)$. By exchanging $S$ and $T$ we can assume that $\sharp \parti(d,S)>\sharp \parti(d,T)$. By the induction hypothesis we know that $\sharp \parti(d_1c,S) =  \sharp \parti(d_1c,T)$ and thus
\begin{eqnarray*}
\sharp  \parti(d_1c,T)&=&\sharp \parti(d_1c,S)\\
&\geq& \sharp \parti(d_1cd_2,S)\\
&>& \sharp \parti(d_1cd_2,T)
\end{eqnarray*}
This shows that $\parti(d_1c,T)-\parti(d_1cd_2,T)$ is non-empty, which implies that there must be a string $d_3\neq d_2$ of length one such that $d_1cd_3$ is a string. As $d_1c$ is non-trivial, there is a most one arrow $\alpha$ such that $d_1c \alpha$ is a string and at most one arrow $\beta$ such that $d_1c\beta^{-1}$ is a string. Consequently, such arrows $\alpha$ and $\beta$ exist and satisfy $\{\alpha,\beta^{-1}\}=\{d_2,d_3\}$. We obtain
\begin{eqnarray*}
\sharp \parti(d_1cd_2,S)+ \sharp \parti(d_1cd_3,S)&=&  \sharp \parti(d_1c,S)  \\
&=&\sharp \parti(d_1c,T)\\
&=&\sharp \parti(d_1cd_2,T)+\sharp \parti(d_1cd_3,T)
\end{eqnarray*}
and thus $\sharp \parti(d_1cd_3,S) \neq  \sharp \parti(d_1cd_3,T)$. If $d_1$ is an arrow, then  $\sharp \fac(c,S)\neq \sharp \fac(c,T)$
and if $d_1^{-1}$ is an arrow, then $\sharp \sub(c,S) \neq \sharp \sub(c,T)$. This gives a contradiction in any case.
\end{proof}

\begin{proof}[Proof of Proposition \ref{separate}]
By the definition of $\mathcal S$-families of band modules it suffices to show that
\[
\sharp \{1 \leq i \leq m: [b_i]=[b]\}=\sharp\{1 \leq i \leq n:[c_i]=[b]\}
\]
for any band $b$. We first show that $\parti(d,S)=\parti(d,T)$ for any non-trivial string $d$.
In order to apply Lemma \ref{techlem}, we need to show that $\rank X(\alpha)=\rank Y(\alpha)$ for any arrow $\alpha$, $X \in \mathcal F(S)$ and $Y \in \mathcal F(T)$. Let $\alpha$ be an arrow. Let $p$ and $q$ be the paths of maximal length such that $q\alpha$ and $\alpha p^{-1}$ are strings.

Note that $M(q\alpha p^{-1})=P_{s(\alpha)}$, where $P_{s(\alpha)}$ is the indecomposable projective module corresponding to the vertex $s(\alpha)$,
and that $M(p^{-1})$ is the cokernel of the morphism $P_{t(\alpha)} \longrightarrow P_{s(\alpha)}$. Applying $\Hom_{\mathcal A}(-,X)$, we obtain the exact sequence
\[
\begin{xy}
\xymatrix{
0 \ar[r]&\Hom_{\mathcal A}(M(p^{-1}),X)\ar[r] & \Hom_{\mathcal A}(P_{s(\alpha)},X) \ar[d]_{\sim}\ar[r] &\Hom_{\mathcal A}(P_{t(\alpha)},X)\ar[d]^{\sim}\\
&& \im X(1_{s(\alpha)}) \ar[r]_{X(\alpha)}& \im X(1_{t(\alpha)})
}
\end{xy}
\]
which shows that $\rank X(\alpha)=[P_{s(\alpha)},X]-[M(p^{-1}),X]$ and thus
\begin{eqnarray*}
\rank X(\alpha)&=&[P_{s(\alpha)},X]-[M(p^{-1}),X]\\
&=&[P_{s(\alpha)},Y]-[M(p^{-1}),Y]\\
&= &\rank Y(\alpha).
\end{eqnarray*}
Applying Lemma \ref{techlem}, we obtain
$\sharp \parti(d,S)= \sharp \parti(d,T)$ for any non-trivial string $d$, as desired.

Let $b=(b,l)$ be a band. For any $k \geq 1$ we define the string
\[
d_k \defeq (b,kl)=b(1)b(2) \cdots b(kl).
\]
Clearly $\sharp \parti(d_k,(b,l))=1$ for any $k\geq 1$ and if $(c,j)$ is a band with $(c,j)  \nsim (b,l)$, then
\[
\sharp \parti (d_k,(c,j))=0
\]
for sufficiently large $k$. We choose $K \in \N$, such that
\[
\sharp \parti (d_K,x)=\begin{cases}
1 &\text{if }  x \sim b\\
0 & \text{otherwise.}
\end{cases}
\]
for any $x \in \{b_1,\ldots,b_m,c_1,\ldots,c_n\}$. We obtain the desired equality
\begin{eqnarray*}
\sharp \{1 \leq i \leq m: [b_i]=[b]\}&=&\sharp \parti(d_K,S)\\
&=&\sharp \parti(d_K,T)\\
&=&\sharp\{1 \leq i \leq n:[c_i]=[b]\}
\end{eqnarray*}
\end{proof}

\section{Proof of Proposition \ref{prop3} and \ref{propdim}}\label{proofof3anddim}

In this section we assume that $\mathcal A=kQ/I$ is a string algebra such that $I$ is generated by a set of paths of length two.
For the proof of Proposition \ref{prop3} we need the following characterization of negligibility.

\begin{lem}\label{charnegligible} The equivalence class of a band $b=(b,m)$ is negligible if and only if the following holds: There are a string $c$ and arrows $\alpha,\beta,\gamma,\delta$ such that the sets
$\parti(\alpha^{-1}c\beta,b)$ and $\parti(\gamma c \delta^{-1},b)$ are non-empty and  $\alpha^{-1} c\delta^{-1}$ and $\gamma c \beta$ are strings.
\end{lem}

\begin{proof} If $(b,m)$ is negligible, one can find a string $c$ and arrows $\alpha,\beta,\gamma,\delta$ such that the sets
$\parti(\alpha^{-1}c\beta,b)$ and $\parti(\gamma c \delta^{-1},b)$ are non-empty and  $\alpha^{-1} c\delta^{-1}$ and $\gamma c \beta$ are strings, by a simple case-by-case analysis which we omit. Indeed, the choice $c=w$ will work in both cases.

We now assume that there are a string $c$ and arrows $\alpha,\beta,\gamma,\delta$ such that the sets
$\parti(\alpha^{-1}c\beta,b)$ and $\parti(\gamma c \delta^{-1},b)$ are non-empty and  $\alpha^{-1} c\delta^{-1}$ and $\gamma c \beta$ are strings, and we want to show that $(b,m)$ is negligible. Up to replacing $(b,m)$ by an equivalent band, we may assume that $m \in \parti_1(\alpha^{-1} c \beta,(b,m))$. We choose $n \in \parti(\gamma c \delta^{-1},(b,m))$. There are two cases to consider:
\begin{itemize}
\item $n \in \parti_{1}(\gamma c \delta^{-1},(b,m))$
\item $n \in \parti_{-1}(\gamma c \delta^{-1},(b,m))$

\end{itemize}
If $n \in \parti_{1}(\gamma c \delta^{-1},(b,m))$, we set 
\[
w=c, u=b(1)\cdots b(n-1), v=b(n+1)\cdots b(m-1).
\]
We have
\[
(b,m)=u\gamma v \alpha^{-1} \text{ and } (b,sm)=w\beta x\alpha^{-1}=u\gamma w \delta^{-1}y
\]
for an integer $s\geq 1$ and strings $x,y$. From now on we use that $I$ is generated by paths of length two.
 As $\gamma c \delta^{-1}$ is a string, we see that $\gamma b(1)$ is a string and thus $\gamma u$ is a string as well. As $\gamma u$ and $u \gamma$ are both strings, we obtain that $u\gamma$ is a quasi-band. Similarly one can show that $v\alpha^{-1}$ is a quasi-band.

We now assume that $n \in \parti_{-1}(\gamma c \delta^{-1},(b,m))$. By the definition of the sets $\parti_1$ and $\parti_{-1}$ we have
\[
b(i)=b(n+l(c)+1-i)^{-1}
\]
for $1 \leq i \leq l(c)$. Note that $l(c)<n$, as otherwise 
\[
\beta = b(l(c)+1)=b(n)^{-1}=\delta^{-1}.
\]
 Similarly we obtain that $n+l(c) <m$. Thus the band $(b,m)$ is of the form
\[
(b,m)=c u c^{-1} v
\]
 for some non-trivial strings $u$ and $v$ of length $l(u)=n-l(c)$ and $l(v)=m-n-l(c)$. As $\beta = b(l(c)+1)$ and $\delta = b(n)$, we see that $u$ starts and ends with an arrow. Similarly we obtain that $v$ starts and ends with an inverse arrow. In order to show that $(b,m)$ is negligible, it remains to prove that
\[
(d,m) \defeq c u^{-1} c^{-1} v
\]
is a quasi-band. As $I$ is generated by paths of length two, it suffices to show that
\[
 u^{-1} c^{-1} v \text{ and } v c u^{-1}
\]
are strings.  We show that $v c u^{-1}$ is a string. We decompose $v=v' \alpha^{-1}$ and $u=u'\delta$. As 
\[
v' \alpha^{-1} \text{ and } \alpha^{-1} c \delta^{-1} \text{ and } \delta^{-1} (u')^{-1} 
\]
are strings, we obtain that 
\[
vcu^{-1}=v'\alpha^{-1} c \delta^{-1} (u')^{-1}
\]
is a string. Similarly one can show that $u^{-1} c^{-1} v$ is a string.
\end{proof}

\begin{proof}[Proof of Proposition \ref{prop3}]
Let $b=(b,m)$ be a band such that the closure of $\mathcal F(b)\subseteq \mod(\mathcal A,d)$ is not an irreducible component of $\mod(\mathcal A,d)$. We assume that $(b,m)$ is not negligible and want to obtain a contradiction.

As $\mathcal F(b)$ is irreducible it must be contained in a irreducible component $\mathcal C$, which is regular as it contains $\mathcal F(b)$. Thus there is a sequence $S=(b_1,\ldots,b_n)$ of bands such that the closure of $\mathcal F(S)$ is $\mathcal C$. As the function $X \mapsto \rank X(\alpha)$ is lower semi-continuous on $\mod(\mathcal A,d)$, we see that $\rank X(\alpha) \leq \rank Y(\alpha)$
for any arrow $\alpha$, any $X \in \mathcal F(b)$ and any $Y \in \mathcal F(S)$. On the other hand,
\[
\sum\limits_{\alpha \in Q_1}\rank X(\alpha)=r(X)=d=r(Y)=  \sum\limits_{\alpha \in Q_1}\rank Y(\alpha)
\] and thus
\[
\sharp \parti (\alpha,b)=\rank X(\alpha) = \rank Y(\alpha)=\sharp  \parti(\alpha,S).
\]
For any string $c$ the functions $[M(c),-]$ and 
$[-,M(c)]$ from $\mod(\mathcal A,d)$ to $\N$ are upper semi-continuous and thus
\[
[c,b] \geq [c,S] \text{ and } [b,c] \geq [S,c].
\]
As we know by Proposition \ref{separate} that strings separate $\mathcal S$-families of bands, there is a string $c$ of minimal length with the property that $[c,b]>[c,S]$ or $[b,c]>[S,c]$. 
We only examine the case $[c,b]>[c,S]$, as the other case is treated similarly.
It follows from Corollary \ref{calcdimhomseq} that $\sharp\sub(c,b) > \sharp\sub(c,S)$.
Thus there are arrows $\alpha,\beta$ such that $\alpha^{-1}c\beta$ is a string and
\[
\sharp \parti(\alpha^{-1}c\beta,b)>\sharp \parti(\alpha^{-1}c\beta,S).
\]
By Lemma \ref{techlem} $\sharp \parti(c\beta,b)=\sharp \parti(c\beta,S)$ and therefore there is an arrow $\gamma$ such that $\gamma c  \beta$ is a string and
\begin{eqnarray*}
\sharp \parti(\alpha^{-1}c\beta,b)+ \sharp \parti(\gamma c\beta,b)&=&\sharp \parti(c\beta,b)  \\
&=&\sharp \parti(c\beta,S)\\
&=&\sharp \parti(\alpha^{-1}c\beta,S)+\sharp \parti(\gamma c\beta,S).
\end{eqnarray*}
In particular $ \sharp \parti(\gamma c \beta,b)<\sharp \parti(\gamma c\beta ,S).$
Similarly we find an arrow $\delta$ such that $\alpha^{-1}c \delta^{-1}$ is a string and $\sharp \parti(\alpha^{-1}c\delta^{-1} ,b)<\sharp \parti(\alpha^{-1}c\delta^{-1},S)$.
We obtain 
\[
\sharp \parti( \gamma c,b)-\sharp \parti(\gamma c \beta,b)> \sharp \parti(\gamma c,S)-\sharp \parti(\gamma c\beta,S)\\
\geq 0.
\]
Hence the word $\gamma c \delta^{-1}$ has to be a string and satisfies
\[\sharp \parti(\gamma c \delta^{-1},b)=\sharp \parti(\gamma c,b)-\sharp \parti(\gamma c \beta,b)>0.\]
From the characterization of negligibility in Lemma \ref{charnegligible} we obtain that $[b]$ is not negligible.
\end{proof}

For the proof of the dimension formula Proposition \ref{propdim} we need another lemma.

\begin{lem}\label{charextendable} Let $(b,m)$ and $(c,n)$ be bands. The pair $[(b,m),(c,n)]$ is extendable if and only if the following holds: There are a string $d$ and arrows $\alpha,\beta,\gamma,\delta$ such that the sets
$\parti(\alpha^{-1}d\beta,(b,m))$ and $\parti(\gamma d \delta^{-1},(c,n))$ are non-empty and  $\alpha^{-1} d\delta^{-1}$ and $\gamma d \beta$ are strings.
\end{lem}

\begin{proof} Set $w=d$ in the definition of extendability and observe that $(d,n+m)$ is a quasi-band if and only if $\alpha^{-1} d\delta^{-1}$ and $\gamma d \beta$ are strings.
\end{proof}

\begin{proof}[Proof of Proposition \ref{propdim}]
Let $S=(b_1,\ldots,b_n)$ be a sequence of bands such that  the closure of $\mathcal F(S)$ is an irreducible component in $\mod(\mathcal A,d)$. From the first part of the main theorem we know that
\begin{itemize}
\item  $([b_i],[b_j])$ is not extendable for $i \neq j$ and
\item $[b_i]$ is not negligible for all $i$.
\end{itemize}
 Let $M=M(b_1,\lambda_1)\oplus \cdots \oplus M(b_n,\lambda_n) \in \mathcal F(S)$ such that  $M(b_1,\lambda_1),\ldots,M(b_n,\lambda_n)$ are pairwise non-isomorphic. The dimension of $\ov{\mathcal F(S)}$ is given by the formula
\[
\dim \ov{\mathcal F(S)} = d^2 +n -[M,M],
\]
as $\mathcal F(S)$ is an $n$-parameter family orbits of dimension $d^2-[M,M]$.
Let $N$ be the set of all tuples $(i,j,k,l,c)$ consisting of integers
$i,j,k,l$ and a string $c$ such that
$i \in \fac(c,b_k)$ and $j \in \sub(c,b_l)$.
By Proposition \ref{band->band} a basis of the space $\Hom_{ \mathcal A}(M,M)$ is given by
\begin{itemize}
\item $\sharp N$ morphisms corresponding tuples $(i,j,k,l,c) \in N$ 
\item $n$ morphisms corresponding to the identities on $M_i$ for $i=1,\ldots,n$. 
\end{itemize}
We thus obtain
\[
\dim \ov{\mathcal F(S)} = d^2 -\sharp N
\]
It remains to show that the cardinality of $N$ is 
\[\sharp N=\sum\limits_{\substack{u \in Q_0 \\ u \text{ non-gentle }}}[X,M(1_u)][M(1_u),X]\] for any $X \in \mathcal F(S).$
Let $(i,j,k,l,c) \in N$. There are arrows $\alpha,\beta,\gamma,\delta$ such that the sets $\parti(\alpha^{-1}c\beta,b_k)$ and $\parti(\gamma c \delta^{-1},b_l)$ are non-empty. Applying Lemma \ref{charextendable} in case $k \neq l$ and Lemma \ref{charnegligible} in case $k=l$, we obtain that at least one of the words
$\alpha^{-1}c \delta^{-1}$ and $\gamma c \beta$ cannot be a string. This can only happen if $c$ is trivial. Let $u$ be the vertex of $Q$ with $c=1_u$:

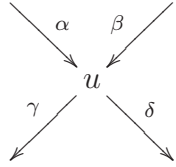
\[
\begin{xy}
\xymatrix{
\ar[dr]^{\alpha}&&\ar[dl]_/0.2em/{\beta}\\
&u\ar[dl]_{\gamma}\ar[dr]^{\delta}&\\
&&
}
\end{xy}
\]

We apply the same lemmas once again: As the sets $\parti(\beta^{-1}\alpha,b_k)$ and $\parti(\gamma  \delta^{-1},b_l)$ are non-empty, we obtain that 
at least one of the words
$\beta^{-1} \delta^{-1}$ and $\gamma \alpha$ cannot be a string.
Therefore none of the pairs of words
\begin{itemize}
\item $(\delta\alpha, \gamma\beta)$
\item $(\gamma\alpha, \delta\beta)$
\end{itemize}
can be a pair of strings. But this is only possible if the vertex $u$ is non-gentle. For the cardinality of $N$ we thus obtain
\begin{eqnarray*}
\sharp N&=&\sum\limits_{\substack{u \in Q_0 \\ u \text{ non-gentle }}} \sharp\fac(1_u,S) \sharp \sub(1_u,S)\\
&=& \sum\limits_{\substack{u \in Q_0 \\ u \text{ non-gentle }}}[S,1_u][1_u,S]\\
&=&\sum\limits_{\substack{u \in Q_0 \\ u \text{ non-gentle }}}[X,M(1_u)][M(1_u),X]
\end{eqnarray*}
for any $X \in \mathcal F(S)$, which completes the proof.
\end{proof}

\newpage
\section{Regular components of indecomposable\\ modules} \label{deg}
An $\mathcal A$-module $Y \in \mod(\mathcal A,d)$ is called a degeneration of $X \in \mod(\mathcal A,d)$ if $Y$ belongs to the closure of the $\GL_d(k)$-orbit of $X$
in $\mod(\mathcal A,d)$. In that case we also say that $X$ degenerates to $Y$ and write $X \leq_{\deg} Y$. We extend this notion to sequences of strings and quasi-bands: Let $S$ and $S'$ be finite sequences of strings and quasi-bands. We call $S$ and $S'$ equivalent, denoted by $S=_{\deg}S'$,  if $\ov{\mathcal F(S)}=\ov{\mathcal F(S')}$,  and
we say that $S$ degenerates to $S'$, in symbols $S \leq_{\deg}S'$,  if $\mathcal F(S') \subseteq  \overline{\mathcal F(S)}$. Note that $\leq_{\deg}$ defines a partial order on the set of equivalence classes of finite sequences of strings and quasi-bands.
Two sequences of strings and bands 
\[
S=(c_1,\ldots,c_l,b_1,\ldots,b_n) \text{ and } S'=(c_1',\ldots,c_{l'}',b_1',\ldots,b_{n'}')
\]
 are equivalent if and only if $l=l'$, $n=n'$ and there are permutations $\sigma \in S_l$ and $\tau \in S_n$ satisfying
\begin{itemize}
\item $c_{\sigma(i)}' \in \{c_i,c_i^{-1}\}$ for $i=1,\ldots,l$ and
\item $b_{\tau(j)}' \sim b_{j}$ for $j=1,\ldots,n$.
\end{itemize}
Note that this characterization might also hold for sequences of strings and quasi-bands, but we do not need it.

We establish two types of degenerations between sequences of bands, which yield a proof for Proposition \ref{prop2}.
\begin{proof}[Proof of Proposition \ref{prop2}]
 Let $(b,m)$ be a band such that the closure of 
\[
\mathcal F(b,m) \subseteq\mod(\mathcal A,d)
\]
 is an irreducible component. If we assume that $(b,m)$ is negligible, we can apply one of the following degenerations and obtain that $\mathcal F(b,m)$ is contained in the closure of another $\mathcal S$-family of quasi-band modules, which is impossible as $\ov{\mathcal F(b,m)}$ is an irreducible component.
\end{proof}

The first degeneration can be described as follows: Cut off a piece of a suitable quasi-band, reverse the piece and reconnect it:
\[
\begin{xy}
\xymatrix{
(b,m)=&\ar[d]_{u}&\ar[l]_{w^{-1}}&\ar@{~>}[r]&  &(c,m)=&  \ar[d]_{u}&\ar[l]_{w^{-1}}\ar[d]_{v}\\
&\ar[r]_{w}&\ar[u]_{v}& && &\ar[r]_{w}&\\
}
\end{xy}
\]

\begin{prop}\label{deg1-1} Let $(b,m)$ be a quasi-band and assume that
there is a string $u$ that starts and ends with an arrow,  a string $v$ that starts and ends with an inverse arrow and a string $w$ such that $(b,m)=wuw^{-1}v$ and
\[
(c,m) \defeq w u w^{-1} v^{-1}
\]
is a quasi-band. Then $(c,m) <_{deg} (b,m)$.

\end{prop}

\begin{proof} Let $\tilde Q$ be the quiver
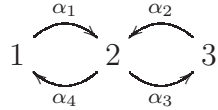
\[
\begin{xy}
\xymatrix{
1 \ar@/^1em/[r]^{\alpha_1} & 2 \ar@/^1em/[l]^{\alpha_4} \ar@/_1em/[r]_{\alpha_3}  &3 \ar@/_1em/[l]_{\alpha_2} \\
}
\end{xy}
\]
and $\mathcal V$ be the variety of representations $X=(X(\alpha_1),X(\alpha_2),X(\alpha_3),X(\alpha_4))$ of $\tilde Q$ with dimension vector $(1,2,1)$, i.e.
\[
\mathcal V = \Mat(2 \times 1,k) \times \Mat(2 \times 1,k)\times \Mat(1 \times 2,k)\times \Mat(1 \times 2,k).
\]
For $\lambda,\mu \in k, \lambda \neq 0$, consider the representations
\[
X_{\lambda,\mu} =(\left(\begin{array}{c}\lambda^{-1} \\ \mu \end{array}\right), \left(\begin{array}{c}1 \\ 0 \end{array}\right),
\left(\begin{array}{cc}0 & 1 \end{array}\right),\left(\begin{array}{cc}-\lambda\mu & 1 \end{array}\right)) \in \mathcal V,
\]
\[
Y_\nu \defeq (\left(\begin{array}{c}0\\ \nu^{-1} \end{array}\right), \left(\begin{array}{c}1 \\ 0 \end{array}\right),
\left(\begin{array}{cc}0 & 1 \end{array}\right),\left(\begin{array}{cc}1 & 0 \end{array}\right)) \in \mathcal V.
\]
The algebraic group $G=k^* \times \GL_2(k) \times k^*$ acts on $\mathcal V$ in the usual way, i.e.
\[
(\varphi,\chi,\psi) \star (X_1,X_2,X_3,X_4) \defeq (\chi X_1 \varphi^{-1}, \chi X_2 \psi^{-1}, \psi X_3 \chi^{-1}, \varphi X_4 \chi^{-1}).
\]
For $\lambda,\mu \in k^*$ we apply the base change
\[
g=(-\lambda^{-1}\mu^{-1},\left(\begin{array}{cc}1 & - \lambda^{-1}\mu^{-1} \\0&1\end{array}\right),1)
\]
to $X_{\lambda,\mu}$ and obtain 
$
g \star X_{\lambda,\mu} =Y_{-\lambda^{-1}\mu^{-2}}.
$
Thus $X_{\lambda,\mu}$ belongs to a $G$-orbit of $Y_\nu$ for some  $\nu \in k^*$, as long as $\mu \neq 0$.

Let $A$ be the set of all paths of $Q$ of length at most one, i.e. $A= Q_1 \cup \{1_{x}: x \in Q_0\}$. We identify the affine variety $\mod(\mathcal A,m)$ with a subvariety of
$M_m(k)^{A}$. To show that $M(b,m,\lambda)$ belongs to the closure of $\mathcal F(c,m)$,
we define a morphism 
\[
\phi: \mathcal V \longrightarrow M_m(k)^{A}
\]
satisfying $\phi(X_{\lambda,\mu}) \in \mathcal F(c,m)$ for $\mu \neq 0$ and $\phi(X_{\lambda,0}) \simeq M(b,m,\lambda)$.
Note that we will define $\phi$ in such a way that $\phi(\mathcal V) \subseteq \mod(kQ,m)$.

Here is the definition of $\phi$:
Let
\[
X=(X(\alpha_1),X(\alpha_2),X(\alpha_3),X(\alpha_4)) \in \mathcal V
\] and
let $Z$ be the $\mathcal A$-module
$
Z = V \oplus M(w)^2 \oplus U,
$
where 
\[
V= \begin{cases} 0 &\text{if } l(v)=1\\
M(v')&\text{if } v=\beta^{-1}v'\alpha^{-1}
\end{cases}
\]
and
\[
U= \begin{cases} 0 &\text{if } l(u)=1\\
M(u')&\text{if } u=\gamma u' \delta
\end{cases}
\]
We decompose $U,V$ and $M(w)^2$ as $k$-vector spaces:
\[
M(w)^2=\bigoplus\limits_{d \in \Ld(w)} W_d,
\]
where $W_d = \spani_{k}\{(e_d,0),(0,e_d)\}$,
\[
V=M(v')=V_{1_{t(v')}} \oplus \cdots \oplus V_{v'}
\]
in case $v=\beta^{-1}v' \alpha ^{-1}$, where $V_{d} \defeq \spani_k \{e_d\}$ for $d \in \Ld(v')$ and
\[
U=U_{1_{t(u')}} \oplus \cdots \oplus U_{u'}
\]
in case $u=\gamma u' \delta$, where $U_{d}=\spani_k\{ e_d\}$ for $d \in \Ld(u')$.

We identify $M_m(k)$ with $\End_k(Z)$, each $W_d$ with $k^2$ and each $U_d$ and $V_d$ with $k$ and set $\phi(X)=Z + Z_v + Z_u$, where the definition of $Z_v,Z_u \in \End_k(Z)$ depends on the lengths of $u$ and $v$.
\\

$
\begin{xy}
\xymatrix{
\text{Case } v=\beta^{-1} v' \alpha^{-1}:&V_{v'}\ar@{.}[dd]\ar[dr]^{Z_v(\alpha)=X(\alpha_1)} \\
&& W_{1_{w}}\ar[ld]^{Z_v(\beta)=X(\alpha_4)}\ar@{.}[r] & W_{w}  & U\\
&V_{1_{t(v')}}\\
}
\end{xy}
$
\\

$
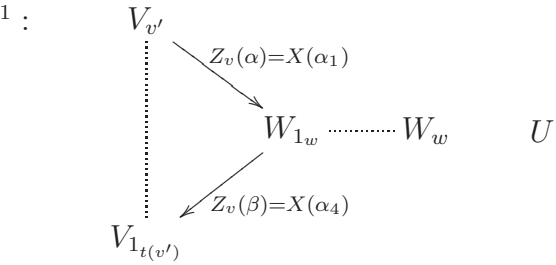
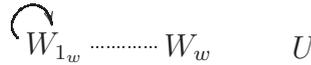
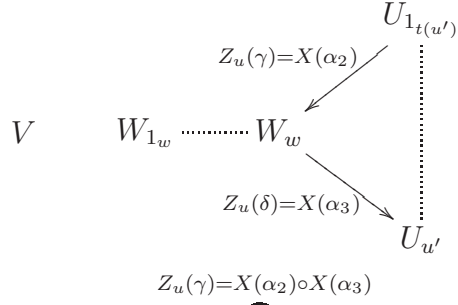
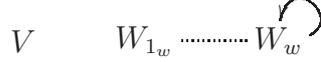
\begin{xy}
\xymatrix{
\text{Case } v= \alpha^{-1}: && W_{1_{w}}\ar@(l,u)^/0.5em/{Z_v(\alpha)=X(\alpha_1)\circ X(\alpha_4)}\ar@{.}[r] & W_{w}  & U\\
&&&&\\
}
\end{xy}
$
\\

$
\begin{xy}
\xymatrix{
\text{Case } u=\gamma u' \delta: & &&& U_{1_{t(u')}} \ar[dl]_{Z_u(\gamma)=X(\alpha_2)}\ar@{.}[dd]\\
& V & W_{1_{w}}\ar@{.}[r] & W_{w}  \ar[rd]_{Z_u(\delta)=X(\alpha_3)}\\
& & & & U_{u'}\\
}
\end{xy}
$
\\

$
\begin{xy}
\xymatrix{
\text{Case } u= \gamma: & V& W_{1_{w}}\ar@{.}[r] & W_{w} \ar@(r,u)_/0.5em/{Z_u(\gamma)=X(\alpha_2)\circ X(\alpha_3)} \\
}
\end{xy}
$
\\

By the definition of $\phi$ we have:
\begin{itemize}
\item $\phi(X_{\lambda,0} )\simeq M(b,m, \lambda) \in \mathcal F(b,m)$ for any $\lambda \in k^*$.
\item $\phi(Y_\nu ) \in \mathcal F(c,m)$ for any $\nu \in k^*$.
\end{itemize}
The morphism $\phi$ is $G$-equivariant with respect to the morphism of algebraic groups
$G \longrightarrow \GL_m(k)$
sending $(\varphi,\chi,\psi)$ to
\[
\left(\begin{array}{ccccc}
\varphi \cdot 1_V & &&&\\
& \chi &&&\\
&&\ddots&&\\
&&&\chi&\\
&&&&\psi \cdot 1_U
\end{array}\right).
\]
Therefore $\phi(X_\lambda,\mu)$ belongs to $\mathcal F(c,m)$ for $\mu \neq 0$ and thus $M(b,m,\lambda)\simeq \phi(X_{\lambda,0})$ belongs to the closure of $\mathcal F(c,m)$ for any $\lambda \in k^*$.

To complete the proof we show that
\[
\ov{\mathcal F(b,m)} \neq \ov{\mathcal F(c,m)}.
\]
As $\sharp\sub(w,(b,m)) \neq \sharp \sub(w,(c,m))$, there is a string $a$
with $[a,(b,m)] \neq [a,(c,n)]$ and thus minimum of the function 
\[
[M(a),-]: \mod(\mathcal A,d) \longrightarrow \N
\]
on $\ov{\mathcal F(b,m)}$ differs from the minimum on $\ov{\mathcal F(c,m)}$, which implies that these two sets cannot be equal.
\end{proof}

The second degeneration can be described as follows: Cut a suitable quasi-band into two pieces, and close each piece to separate quasi-bands.
\[
\begin{xy}
\xymatrix{
(b,m) &\ar@{~>}[r]&&(c,n)&&(d,m-n)\\
\ar[d]_{u}&\ar[l]_{\gamma}       && \ar[d]_{u}& &\ar@/_1em/[d]_{\alpha^{-1}} \\
\ar[r]_{\alpha^{-1}}&\ar[u]_{v} && \ar@/_1em/[u]_{\gamma}& &\ar[u]_{v}\\
}
\end{xy}
\]

\begin{prop} Let $(b,m)$ be a quasi-band and assume that
there are strings $u,v,w,x,y$, arrows $\alpha,\beta,\gamma,\delta$ and an $s\geq 1$ with 
\[(b,m)=u\gamma v \alpha^{-1} \text{ and }
(b,sm)=w \beta x \alpha^{-1}=u \gamma w \delta^{-1} y\]
such that
\[(c,n) \defeq u\gamma \text{ and }  (d,m-n) \defeq v\alpha^{-1}\] are quasi-bands. Then $((c,n),(d,m-n)) <_{\deg} (b,m)$.
\end{prop}

\begin{proof}
We only show that $((c,n),(d,m-n)) \leq_{\deg} (b,m)$ and leave the proof of the inequality $((c,n),(d,m-n)) \neq_{\deg} (b,m)$ to the reader, as it is nearly the same as in the proof of Proposition \ref{deg1-1}.

Let $\mu,\nu \in k^*$ and set $M \defeq M(d,m-n,\mu)$ and $N \defeq M(c,n,\nu)$. For any $h \in \Hom_{k}(M,N)$ there is a unique $\mathcal A$-module structure $X_{h,\mu,\nu}$ on the vector space $N\oplus M$ such that
\[
\left(\begin{array}{cc}
1_N & h\\
0 & 1_M
\end{array}\right): X_{h,\mu,\nu} \longrightarrow N \oplus M
\]
is an $\mathcal A$-isomorphism. By definition, we know that 
\[
X_{h,\mu,\nu}(a)=\left(\begin{array}{cc}
N(a) & \zeta(a)\\
0 & M(a)
\end{array}\right)
\]
with $\zeta(a)=N(a)h-hM(a)$ for any $a \in \mathcal A$. Note that $\zeta(a)=0$ for all $a \in \mathcal A$ in case $h$ is $\mathcal A$-linear.

We will construct an $h \in \Hom_{k}(M,N)$ in such away that
\[
\zeta(\epsilon)(\ov{1\otimes e_i})=\begin{cases}
-\nu^{-1}\mu^{-1} \ov{1\otimes e_{0}} &\text{if }i=m-n-1 \text{ and } \epsilon=\alpha,\\
 \ov{1\otimes e_{n-1}} &\text{if } i=0 \text{ and }\epsilon=\gamma ,\\
0 & \text{otherwise,} \end{cases}
\]
for $0 \leq i <m-n$ and for any path $\epsilon$ of length at most one.
Here is a picture of the module $X_{\mu,\nu} \defeq X_{h,\mu,\nu}$:

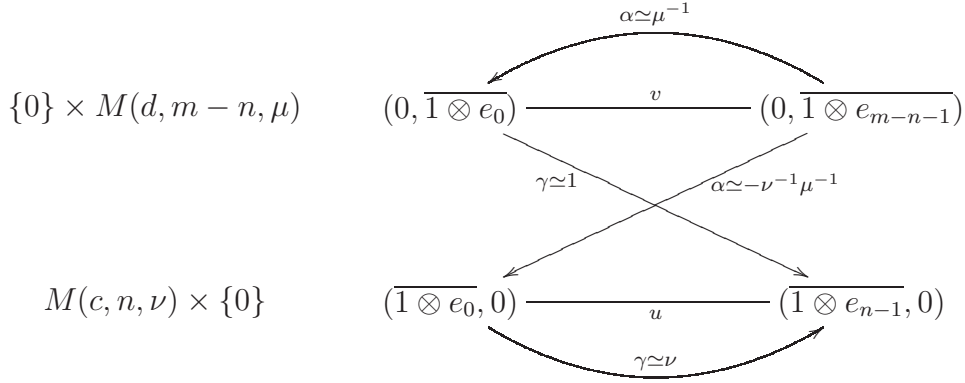
\[
\begin{xy}
\xymatrix{
\{0\}\times M(d,m-n,\mu)&(0,\ov{1\otimes e_{0}})\ar[ddrrr]_/-3em/{\gamma \simeq 1}\ar@{-}[rrr]^{v} &&&(0,\ov{1\otimes e_{m-n-1}})\ar@/_1cm/[lll]_{\alpha \simeq \mu^{-1}}\ar[ddlll]^/-3em/{\alpha \simeq -\nu^{-1}\mu^{-1}}\\
\\
M(c,n,\nu)\times \{0\}&(\ov{1\otimes e_0},0)\ar@/_1cm/[rrr]^{\gamma \simeq \nu}&&  &(\ov{1\otimes e_{n-1}},0) \ar@{-}[lll]^{u} \\
}
\end{xy}\]
(An arrow from $x$ to $y$ labeled by $\alpha \simeq\lambda$ indicates that $X_{h,\mu,\nu}(\alpha)(x)=\lambda y$.)

We postpone the construction of $h$ and show how to complete the proof, once $h$ is defined:
For a fixed $\lambda \in k^*$ the module $M(b,m,\lambda)$ belongs to the closure of the one-parameter family $Y_{\mu}\defeq X_{\mu,-\lambda\mu^{-1}}$, $\mu \in k^*$. Consequently, \[((c,n),(d,m-n)) \leq_{deg} (b,m).\]

In order to define $h$, we have to show that
\[c(1)\cdots c(l(w)+1)=w\beta=b(1)\cdots b(l(w)+1).\]
Let $0<i \leq l(w)+1$. Obviously $c(i)=b(i)$ if $i\leq n$. If $n<i$, we may assume inductively that $c(i-n)=b(i-n)$ and thus
\[
c(i)=c(i-n)=b(i-n)=b(i).
\]
Similarly, one can show that \[d(1)\cdots d(l(w)+1)=w\delta^{-1}.\]
There are integers $t,r \geq 0$ and strings $z,z'$ such that
\[
(c,tn)=w \beta z \gamma \text{ and } (d,r(m-n))=w \delta^{-1} z' \alpha^{-1}.
\]
Let $g: M(w) \longrightarrow M(c,tn,\nu^{t})$ be the $k$-linear map sending $e_{d}$ to $\ov{1\otimes e_{l(d)}}$ for $d \in \Ld(w)$. We have
\[
(g \circ \epsilon - \epsilon \circ g)(e_{d})=\begin{cases}-\nu^{t}\ov{1\otimes e_{tn-1}} &\text{if } d=1_{t(w)} \text{ and } \epsilon=\gamma,\\
0 & \text{otherwise,} \end{cases}
\]
for any path $\epsilon$ of length at most one and any $d \in \Ld(w)$. Similarly, let $f:M(d,r(m-n),\mu^r) \longrightarrow M(w)$ be the $k$-linear map sending
$\ov{1\otimes e_i}$ to 
\[
\begin{cases}
e_{d} & \text{if } 0 \leq i \leq l(w),d \in \Ld(w), l(d) =i\\
0 & \text{if } l(w)<i <r(m-n)
\end{cases}
\] for $i=0, \ldots,r(m-n)-1$. The map $f$ satisfies
\[
(f \circ \epsilon - \epsilon \circ f)(\ov{1\otimes e_{i}})=\begin{cases}\mu^{-r}e_{1_{t(w)}} &\text{if } i=tn-1 \text{ and } \epsilon=\alpha,\\
0 & \text{otherwise,} \end{cases}
\]
for any path $\epsilon$ of length at most one and $0 \leq i < tn$. For the composition of $f$ and $g$, we have
\[
(g \circ f \circ \epsilon - \epsilon\circ g \circ f)(\ov{1\otimes e_{i}})=\begin{cases}
\mu^{-r} \ov{1\otimes e_{0}} &\text{if }i=tn-1 \text{ and } \epsilon=\alpha,\\
 -\nu^{t}\ov{1\otimes e_{r(m-n)-1}} &\text{if } i=0 \text{ and }\epsilon=\gamma ,\\
0 & \text{otherwise,} \end{cases}
\]
for any path $\epsilon$ of length at most one and $0 \leq i <r(m-n)$. Let $h'$ be the composition
\[
\begin{xy}
\xymatrix{
M(d,m-n,\mu)\ar[r]& M(d,r(m-n),\mu^r)\ar[r]^/0.8em/{g\circ f}  & M(c,tn,\nu^t)\ar[r]&M(c,n,\nu),
}
\end{xy}
\]
where the first map is an unwinding morphism and the last map is a winding morphism. We have
\[
(h' \circ \epsilon - \epsilon\circ h')(\ov{1\otimes e_{i}})=\begin{cases}
\mu^{-1} \ov{1\otimes e_{0}} &\text{if }i=tn-1 \text{ and } \epsilon=\alpha,\\
 -\nu\ov{1\otimes e_{m-n-1}} &\text{if } i=0 \text{ and }\epsilon=\gamma ,\\
0 & \text{otherwise,} \end{cases}
\]
for any path $\epsilon$ of length at most one and any $0 \leq i < m-n$. Finally, we set $h\defeq -\nu^{-1}h'$.

\end{proof}

\newpage
\section{Extensions}\label{ext}

Let $(b,m)$ and $(c,n)$ be bands. Proposition \ref{prop1} will follow from the following two lemmas.
\begin{lem}\label{notextendable->ext=0}
 If the pair $([(b,m)],[(c,n)])$ is not extendable, then \[\Ext_{\mathcal A}^1(M(b,m,\lambda),M(c,n,\mu))=0\] for any $\lambda,\mu \in k^*$ with $\lambda \neq \mu,\mu^{-1}$.
\end{lem}

\begin{lem} \label{extenable->extneq0}
 If the pair $((b,m),(c,n))$ is extendable, there is a non-split short exact sequence of $\mathcal A$-modules 
\[0 \longrightarrow X \longrightarrow M_{X,Y} \longrightarrow Y \longrightarrow 0\] with $M_{X,Y} \in \mathcal F(d,n+m)$ for any $X \in \mathcal F(c,n)$, $Y \in \mathcal F(b,m)$, where \[(d,n+m)=c(1)\cdots c(n)b(1)\cdots b(m).\]
\end{lem}

Recall from \cite{Riedtmann} that an $\mathcal A$-module $A$ degenerates to a direct sum of $\mathcal A$-modules $B\oplus C$ whenever there is a short exact sequence
\[
0 \longrightarrow B \longrightarrow A \longrightarrow C \longrightarrow 0.
\]
Combining this result with Lemma \ref{extenable->extneq0}, we obtain
\begin{cor} If the pair $((b,m),(c,n))$ is extendable, then
\[\mathcal F((c,n),(b,m)) \subseteq \ov{\mathcal F(d,n+m)},\] where 
\[(d,n+m)=c(1)\cdots c(n)b(1)\cdots b(m).\]
\end{cor}

\begin{proof}[Proof of Lemma \ref{notextendable->ext=0}]
By \cite{ButlerRingel} the Auslander-Reiten translate $\tau^{-1} M$ of a band module $M$ is isomorphic to $M$. It is well known that
\[
\Ext_{\mathcal A}^1(N,M)\simeq\HomP_{\,\mathcal A}(\tau^{-1} M,N)
\]
for any finite dimensional $\mathcal A$-modules $M,N$, where
\[
\HomP_{\,\mathcal A}(\tau^{-1} M,N) =\Hom_{\mathcal A}(\tau^{-1} M,N)/\mathcal P(\tau^{-1} M,N)
\]
and $\mathcal P(\tau^{-1} M,N)$ is the subspace of $\Hom_{\mathcal A}(\tau^{-1} M,N)$ consisting of all homomorphisms that factor through a projective $\mathcal A$-module.
Therefore it suffices to show that any morphism $M(c,n,\mu)\longrightarrow M(b,m,\lambda)$ factors through a projective $\mathcal A$-module. The condition $\lambda \neq \mu,\mu^{-1}$ asserts that $M(c,n,\mu)$ and $M(b,m,\lambda)$ are non-isomorphic and this allows us to apply Proposition \ref{band->band}, which yields a basis for
\[
\Hom_{\mathcal A}(M(c,n,\mu),M(b,m,\lambda)).
\]
We show that any morphism of the form
\[
\iota_{j,w,(b,m),\lambda} \circ \pi_{i,w,(c,n),\mu}:M(c,n,\mu) \longrightarrow M(b,m,\lambda)
\]
with $j \in \sub(w,(b,m))$ and $i \in \fac(w,(c,n))$ factors through a projective $\mathcal A$-module. Fix a string $w$, $j \in \sub(w,(b,m))$ and $i \in \fac(w,(c,n))$.
 Up to equivalence of bands, we may assume that $i=n$ and $j=m$. Moreover, we can assume that $m \in \sub_1(w,(b,m))$ and $n \in \sub_1(w,(c,n))$ up to replacing $\lambda$ by $\lambda^{-1}$ and $\mu$ by $\mu^{-1}$ if necessary.
By the definition of the sets $\sub_1$ and $\fac_1$ there are $s,t \geq 1$, strings $u,v$ and arrows $\alpha,\beta,\gamma,\delta$ with
\[
(b,sm)=w\beta u \alpha^{-1} \text{ and } (c,tn)=w\delta^{-1} v \gamma.
\]
As the pair $((b,m),(c,n))$ is not extendable,
\[
(d,n+m)=c(1) \cdots c(n)b(1) \cdots b(m)
\]
cannot be a quasi-band. As $c(n)=\gamma$ and $b(m)= \alpha^{-1}$ this can only happen if one of the words
\[
c(1) \dots c(n)b(1) \cdots b(m-1) \text{ and } b(1) \cdots b(m) c(1) \cdots c(n-1)
\]
is not a string. We just consider the case, where the word
\[x\defeq c(1) \cdots c(n) b(1) \cdots b(m-1)\]
 is not a string, as the other case is treated similarly. Let $1 \leq i \leq n$ be minimal with the property that $c(i),c(i+1) ,\ldots , c(n)$ are arrows and let $1 \leq j \leq m-1$ be maximal, such that $b(1) ,\ldots, b(j)$ are arrows. As $x$ is not a string, we see that the path
\[
c(i)c(i+1) \cdots c(n) b(1) \cdots b(j)
\]
belongs to the ideal $I$. We set
\[
q=c(i)c(i+1)\cdots c(n) \text{ and } r=b(1) \cdots b(j).
\]
Obviously $w\beta \in \Ld(r)$, because otherwise $r \in \Ld(w)$, which implies that $qr$ is a string. We may thus decompose $r=w r'$ for some non-trivial path $r'$.
Let $P=P_{s(r)}$ be the projective $\mathcal A$-module corresponding to the vertex $s(r)$. Obviously $P$ is isomorphic to $M(q'r p^{-1})$ for some paths $p$ and $q'$ with $q=q'' q'$ for a non-trivial path $q''$. We obtain the sequence of morphisms
\[
\begin{xy}
\xymatrix{
M(c,n,\mu) \ar[r]^{h}& M(q'w) \ar[r]^/-1em/{g} & P=M(q'wr'p^{-1})\ar[r]^/0.8em/{f} & M(b,m,\lambda),
}
\end{xy}
\]
where $h$ is the factorstring morphism corresponding to 
\[
n-l(q') \in\fac_1(q'w,(c,n)),
\]
 $g$ is the substring morphism corresponding to the decomposition
\[
q'rp^{-1} = (1_{t(q')})\, (q'w) \, (r'p^{-1})
\] and $f$ is the morphism sending $e_{q'r}$ to $\ov{1 \otimes e_{l(r)}}$. 
As
\[
f \circ g \circ h =\iota_{j,w,(b,m),\lambda} \circ \pi_{i,w,(c,n),\mu},
\] the proof is complete.

\end{proof}

\begin{proof}[Proof of Lemma \ref{extenable->extneq0}] As $((b,m),(c,n))$ is extendable, there are $s,t \geq 1$, strings $u,v,w$ and arrows $\alpha,\beta,\gamma,\delta$ with
\[
(b,sm)=w\beta u \alpha^{-1} \text{ and } (c,tn)=w\delta^{-1} v\gamma
\]
 such that
\[(d,n+m)\defeq(c,n)(b,m) \defeq c(1)\cdots c(n) b(1) \cdots b(m)\] is a quasi-band.

Let $x=c(1)\cdots c(n)$ and $y=b(1)\cdots b(m)$ as strings.
We divide the proof into four steps:
\begin{itemize}
\item[a)] $l(w) < n+m$,
\item[b)] $n+m \in \sub(xw,(d,n+m))$,
\item[c)] $n \in \fac(yw,(d,n+m))$ and
\item[d)] for any $\lambda,\mu \in k^*$ the sequence
\[
\begin{xy}
\xymatrix{
0 \ar[r]& M(c,n,\mu)\ar[r]^/-1em/f & M(d,n+m-\lambda\mu)\ar[r]^/1em/g &M(b,m,\lambda)\ar[r]&0,
}
\end{xy}
\]
where 
\[
f= \iota_{n+m,xw,(d,n+m),-\lambda\mu} \circ \pi_{n,xw,(c,n),\mu}
\]
and 
\[
g= \iota_{m,yw,(b,m),\lambda} \circ \pi_{n,yw,(d,n+m),-\lambda\mu}
\]
is exact and does not split.
\end{itemize}

\textit{Proof of} a): If we assume that $l(w) \geq n+m$, we obtain the contradiction
\[
\alpha^{-1}=b(m)=c(m)=c(n+m)=b(n+m)=b(n)=c(n) =\gamma.
\]

\textit{Proof of} b): As $d(n+m)=b(m)=\alpha^{-1}$ is an inverse arrow and $d(i)=c(i)$ for $i=1,\ldots,n$, it suffices to show that
$d(i+n)=b(i)$ for $i=1,\ldots,l(w)+1$, as $b(l(w)+1)=\beta$ is an arrow. This is obvious if $i \leq m$. We may thus assume that $i>m$. As  $l(w) < n+m$ by a), we see that  $0 < i-m \leq n$ and $i-m \leq l(w)$ and thus
\[
d(i+n)=d(i-m) =c(i-m)=b(i-m)=b(i).
\]

Statement c) follows dually.

\textit{Proof of} d): We decompose $f$ as the sum $\sum\limits_{i=0}^{l(w)+n}f_i$ of $k$-linear maps
\[
f_i:M(c,n,\mu) \longrightarrow M(d,n+m,-\lambda\mu).
\]
Note that, in order to keep the coefficients combinatorially simple, we adapt the basis of $M(c,n,\mu)$ to $i$ for the definition of $f_i$: For $0 \leq i\leq l(w)+n$ and $i \leq l <i+n$ we set
\[
f_i(\overline{1 \otimes e_l})\defeq \begin{cases}
\overline{1\otimes e_i}&\text{if } i=l\\
0&\text{otherwise.}\\
\end{cases}
\]
\noindent
Similarly, we decompose $g$ as the sum $\sum\limits_{k=n}^{l(w)+m+n} g_k$ of $k$-linear maps
\[
g_k: M(d,n+m,-\lambda\mu) \longrightarrow M(b,m,\lambda),
\]
where $g_k$ is defined as follows:
For $n\leq k \leq  l(w)+n+m$ and $k \leq l <k+n+m$ we set
\[
g_k(\overline{1 \otimes e_l})\defeq \begin{cases}
\overline{1\otimes e_{l-n}}&\text{if } k=l\\
0&\text{otherwise.}\\
\end{cases}
\]

Applying Lemma \ref{?->injective}, we see that $f$ is injective and $g$ surjective.
In order to prove that the sequence is exact, it remains to show that $g\circ f=0$. We have
\[
g \circ f= \sum \limits_{i=0}^{l(w)+n}\; \sum \limits_{k=n}^{l(w)+m+n} g_{k} \circ f_{i}.
\]
We claim that $g_{k}\circ f_i =0$ unless $(i,k)$ belongs to one of the disjoint sets
\[
A \defeq \{(i,i+n+m): 0 \leq i \leq  l(w)\}
\]
and 
\[
B \defeq \{(i,i): n \leq i \leq  l(w)+n\}.
\]

Assume that there are $0 \leq i\leq l(w)+n$, $n \leq k \leq n+m+l(w)$ and $0 \leq l <n$ such that $g_k\circ f_i(\ov{1\otimes e_l}) \neq 0$. From the definition of $f_i$ we obtain
$l-i \in n\Z$ and $f_i(\ov{1\otimes e_l})= \xi \ov{(1\otimes e_i)}$ for a $\xi \in k^*$ and from the definition of $g_k$ we see that $k-i \in (n+m)\Z$. As
$-l(w)\leq k-i  \leq n+m+l(w)$ and $l(w) < n+m$, we obtain either $k-i =0$ and thus $(i,k)\in B$ or $k-i=n+m$ and hence $(i,k) \in A$.
We have 
\[
g\circ f= \sum\limits_{(i,k)\in A}g_k\circ f_i + \sum \limits_{(i,k)\in B}g_k\circ f_i = \sum\limits_{(i,k)\in A} (g_k\circ f_i + g_{k-m}\circ f_{i+n})
\]
and suffices to show that $(g_k\circ f_i + g_{k-m}\circ f_{i+n})=0$ for $(i,k) \in A$. Let $(i,k) \in A$. For any $i \leq l<n+i$ we have
\[f_i(\ov{1\otimes e_l})=f_{i+n}(\ov{1 \otimes e_l})=0\] and thus
\[g_k\circ f_i (\ov{1 \otimes e_l}) + g_{k-m}\circ f_{i+n}(\ov{1 \otimes e_l})=0,\]
unless $l=i$. In case $l=i$ we obtain

\begin{eqnarray*}
g_kf_i(\ov{1\otimes e_i})&=& g_k(\ov{1\otimes e_i}) \\
&=& g_k(-\lambda\mu\cdot  \ov{1\otimes e_{i+n+m}})\\
&=& g_k(-\lambda\mu\cdot  \ov{1\otimes e_k})\\
&=& -\lambda\mu\cdot  \ov{1\otimes e_{k-n}}
\end{eqnarray*}
and
\begin{eqnarray*}
g_{k-m}f_{i+n}(\ov{1\otimes e_i})&=&g_{k-m}f_{i+n}(\mu\cdot  \ov{1\otimes e_{i+n}})\\
&=&g_{k-m}(\mu \cdot \ov{1\otimes e_{i+n}}) \\
&=&g_{k-m}(\mu \cdot \ov{1\otimes e_{k-m}}) \\
&=& \mu \cdot \ov{1\otimes e_{k-m-n}}\\
&=& \lambda\mu\cdot   \ov{1\otimes e_{k-n}}
\end{eqnarray*}
and thus $g \circ f=0$.

To complete the proof, it remains to show that the short exact sequence
\[
\begin{xy}\xymatrix{0\ar[r]& M(c,n,\mu)\ar[r]^/-0.8em/{f} &M(d,n+m,-\lambda\mu) \ar[r]^/0.8em/{g}& M(b,m,\lambda) \ar[r]&0,\\
}
\end{xy}
\]
does not split. It suffices to show that $M(d,n+m,-\lambda\mu)$ is not isomorphic to $M(c,n,\mu) \oplus M(b,m,\lambda)$. As
\[
\sharp \sub(w,(c,n)) + \sharp\sub (w,(b,m)) > \sharp\sub(w,(d,n+m)),
\]
there is an $\mathcal A$-module $U$, such that
\[[U,M(d,n+m,-\lambda\mu)]\neq [U,M(c,n,\mu) \oplus M(b,m,\lambda)],\] by Proposition \ref{string->band} about the homomorphism spaces between string and band modules. This shows that 
$M(d,n+m,-\lambda\mu)$ and $M(c,n,\mu) \oplus M(b,m,\lambda)$ cannot be isomorphic.
\end{proof}

\newpage
\section*{{\rm Acknowledgments}}
The author thanks his supervisor, Professor Christine Riedtmann, for her aid and guidance and for insisting on the readability of this paper.

\bigskip

\noindent
Manuel Rutscho\\
Mathematisches Institut\\
Universit\"at Bern\\
Sidlerstrasse 5\\
CH-3012 Bern\\
Switzerland\\
e-mail: manuel.rutscho@math.unibe.ch,\\


\begin{thebibliography}{99}



\bibitem{Auslander-Reiten}  M. Auslander and I. Reiten, Modules determined by their composition factors, Illinois J of Math \textbf{29}, 289 - 301

\bibitem{ButlerRingel} M. C. R. Butler and C. M. Ringel, Auslander-Reiten sequences with few middle terms and applications to string algebras, Comm. Algebra \textbf{15} (1987), 145 - 179. 


\bibitem{CBJS} W. Crawley-Boevey and J. Schr\"oer, Irreducible components of varieties of modules, J. Reine und Angewandte Mathematik \textbf{553} (2002), 201 - 220.


\bibitem{GeissSchroer} C. Gei\ss, J. Schr\"oer, Varieties of modules over tubular algebras, Colloq. Math. \textbf{95} (2003), no. 2,
163 - 183.

\bibitem{Krause} H.Krause, Maps between tree and band modules, J.Algebra \textbf{137} (1991), 186 - 194.


\bibitem{Riedtmann} Ch. Riedtmann, Degenerations for representations of quivers with relations, Ann. Sci. Ecole Norm. Sup. \textbf{19} (1986), 275 - 301.


\bibitem{SchroerGelfandPonomarev}J. Schr\"oer,
Varieties of pairs of nilpotent matrices annihilating each other,
Commentarii Mathematici Helvetici \textbf{79} (2004), 396 -- 426.




\end{thebibliography}
\end{document}